\documentclass[11pt]{amsart}

\usepackage[utf8]{inputenc}
\usepackage{geometry}

\usepackage[nocompress]{cite}

\makeatletter
\@namedef{subjclassname@2020}{%
  \textup{2020} Mathematics Subject Classification}
\makeatother

\usepackage{amsmath, amssymb, bbm}
\usepackage{amsthm, thmtools}

\usepackage[hidelinks,pdfencoding=auto]{hyperref}
\usepackage[capitalise,nosort]{cleveref}

\declaretheoremstyle[bodyfont=\normalfont]{defnstyle}

\declaretheorem[numberwithin=section,refname={def.,defs.},
                Refname={Definition,Definitions},style=defnstyle,name=Definition]{definition}

\declaretheorem[numberlike=definition,refname={cor.,cors.,},
                Refname={Corollary,Corollaries},name=Corollary]{corollary}

\declaretheorem[numberlike=definition,refname={Prop.,Props.},
                Refname={Proposition,Propositions},name=Proposition]{proposition}

\declaretheorem[numberlike=definition,style=defnstyle,name=Counter-example,
                refname={Counter-ex.,Counter-exs.},Refname={Counter-example,Counter-examples}]{ce}

\declaretheorem[numberlike=definition,style=remark]{remark}

\declaretheorem[numbered=no]{claim}

\declaretheorem[numberlike=definition,style=remark]{reduction}

\declaretheorem[numberlike=definition,style=defnstyle]{notation}

\declaretheorem[style=defnstyle,refname={Conjecture,Conjectures,}]{conjecture}

\declaretheorem[numbered=no,style=remark,qed=$\blacksquare$]{proof}

\usepackage{enumitem}

\newlist{rome}{enumerate}{7}
\setlist[rome]{label=(\roman*)}

\usepackage{tikz}

\newcounter{diagram}
\newenvironment{diagram*}{\begin{center}\begin{tikzpicture}}{\end{tikzpicture}\end{center}}
\newenvironment{diagram}{\setcounter{diagram}{\value{definition}}\refstepcounter{diagram}\begin{center}\hfill\begin{tikzpicture}}
    {\end{tikzpicture}\hfill{\normalfont(\thediagram)}\end{center}\refstepcounter{definition}}
\crefname{diagram}{diagram}{diagrams}
\Crefname{diagram}{Diagram}{Diagrams}
\numberwithin{diagram}{section}

\usetikzlibrary{arrows.meta,shapes,arrows,calc}

\tikzset{node distance=1.5cm, auto,
  baseline=(current bounding box.center),
  shrink/.style={outer sep=2pt,inner sep=0, minimum size=0},
  squeeze/.style={auto=false,outer sep=3pt,inner sep=0, minimum size=0},
  a/.style={->}, e/.style={->>},
  u/.style={dashed,->},
  n/.style={double equal sign distance, -implies},
  ne/.style={double equal sign distance, -},
  t/.style={double distance=2.5pt, -implies, postaction={draw,-}},
  te/.style={double distance=2.5pt, -, postaction={draw,-}},
  cover/.style={preaction={draw=white, -,line width=6pt}},
  over/.style={auto=false,fill=white,inner sep=1.5pt, minimum size=0, outer sep=0},
  s/.style={shorten >=3pt, shorten <=3pt},
  la/.style={scale=0.8}}

\def\cellslide{0.5}
\def\celllength{.225cm}

\usepackage{xparse}
\NewDocumentCommand{\cell}{ O{} O{n} O{\cellslide} O{\celllength} m m m }{
  \coordinate (mid) at ($({#5})!{#3}!({#6})$);
  \coordinate (start) at ($(mid)!{#4}!({#5})$);
  \coordinate (end) at ($(mid)!{#4}!({#6})$);
  \draw[#2] (start) to node
  [inner sep=4pt,outer sep=0,minimum size=0,#1]{{#7}} (end);
}

\newdimen{\pullbackmarkerlength}
\def\pullbackangle{30}

\pgfmathsetlengthmacro{\pullbackradius}{\pullbackmarkerlength*sin(45)/sin(45-\pullbackangle)}
\newcommand{\pullback}[1]{
  \draw ({#1}.center) ++(360-\pullbackangle:\pullbackradius) -- +(0,-\pullbackmarkerlength);
  \draw ({#1}.center) ++(270+\pullbackangle:\pullbackradius) -- +(\pullbackmarkerlength,0);
}

\DeclareMathOperator{\id}{id}
\DeclareMathOperator{\Psd}{Ps}
\DeclareMathOperator{\Lax}{Lax}
\newcommand{\iso}{\mathrel{\cong}}
\newcommand{\eqv}{\mathrel{\simeq}}
\newcommand{\Cat}{\textnormal{\textsf{Cat}}}

\newcommand{\slice}[2]{\ensuremath{{#1}\downarrow{#2}}}
\newcommand{\sslice}[3][]{\ensuremath{{#2}\downarrow^{\textnormal{\tiny #1}}_{s}{#3}}}
\newcommand{\pslice}[3][]{\ensuremath{{#2}\downarrow^{\textnormal{\tiny #1}}_{p}{#3}}}
\newcommand{\lslice}[3][]{\ensuremath{{#2}\downarrow^{\textnormal{\tiny #1}}_{l}{#3}}}

\newlength{\arrowlength}
\newcommand{\threecellarr}{\mathrel{\begin{tikzpicture}[baseline=(A.base)]
      \node(A)[inner sep=0,outer sep=0,minimum size=0] at (0,0) {\vphantom{a}};
      \draw[t](A) -- ++(\arrowlength,0);
    \end{tikzpicture}}}

\newcommand{\ainl}[3]{\ensuremath{ {#2}\colon{#1}\rightarrow{#3} }}
\newcommand{\aiinl}[3]{\ensuremath{ {#2}\colon{#1}\xrightarrow{\tiny\iso}{#3} }}
\newcommand{\ninl}[3]{\ensuremath{ {#2}\colon{#1}\Rightarrow{#3} }}
\newcommand{\minl}[3]{\ensuremath{ {#2}\colon{#1}\threecellarr{#3} }}

\usepackage{xspace}
\makeatletter
\newcommand{\@calify}[1]{
  \ifcsname c#1\endcsname
    \message{WARNING: Not over-writing c#1 with calligraphic letter!}
  \fi
    \expandafter\edef\csname c#1\endcsname{\noexpand\ensuremath{\noexpand\mathcal #1}\noexpand\xspace}
  }
\newcounter{@calcount}\stepcounter{@calcount}
\loop\@calify{\Alph{@calcount}}
\ifnum\the@calcount<26\stepcounter{@calcount}\repeat
\makeatother

\newcommand{\valid}{\ensuremath{\checkmark}}

\title{$2$-limits and $2$-terminal objects are too different}

\author[t.\ clingman]{tslil clingman}
\address{Johns Hopkins University, 3400 N. Charles St., Baltimore, MD, USA}
\email{tslil@jhu.edu}

\author[L.\ Moser]{Lyne Moser} \address{UPHESS BMI FSV, Ecole Polytechnique Fédérale de Lausanne, Station 8, 1015 Lausanne, Switzerland}
\email{lyne.moser@epfl.ch}


\begin{document}

\maketitle

\begin{abstract}
  In ordinary category theory, limits are known to be equivalent to terminal objects in the slice category of cones. In this paper, we prove that the 2-categorical analogues of this theorem relating 2-limits and 2-terminal objects in the various choices of slice 2-categories of 2-cones are false. Furthermore we show that, even when weakening the 2-cones to pseudo- or lax-natural transformations, or considering bi-type limits and bi-terminal objects, there is still no such correspondence.
\end{abstract}

\section{Introduction}

In this paper we address the question of whether the natural $2$-categorical analogue to the $1$-categorical result giving a correspondence between limits and terminal objects in a slice category of cones, holds.

\subsection{Motivation, the \texorpdfstring{$1$}{1}-dimensional case and the case of \Cat{}}

A \emph{limit} of a functor $\ainl IF\cC$ comprises the data of an object $L\in \cC$ together with a
natural transformation $\ninl{\Delta L}\lambda F$, called the \emph{limit cone}, which satisfies the following universal property: for each $X\in\cC$, the map
 $\ainl{\cC(X,L)}{\lambda_{*}\circ\Delta}{\Cat(I,\cC)(\Delta X,F)}$
given by post-composition with $\lambda$ is an isomorphism of sets. We may also form the slice category $\slice \Delta F$ of cones over $F$, and it is a folklore result that a limit of $F$ is equivalently a terminal object in $\slice \Delta F$.

As an example, and in progressing up in dimension, let us now consider products in \Cat{} -- the category of small categories and functors. The universal property of the product $(\cC\times \cD,\pi_\cC,\pi_\cD)$ of two categories $\cC$ and $\cD$ gives, for each pair of functors $\ainl \cX F\cC$ and $\ainl \cX G\cD$, a functor $\ainl{\cX}{\left<F,G\right>}{\cC\times\cD}$ unique
among those satisfying $\pi_{\cC}\left<F,G\right>=F$ and $\pi_{\cD}\left<F,G\right>=G$.
\begin{diagram} \label{univproduct}
    \node(1)[]{$\cC$};
    \node(2)[right of= 1,xshift=0.3cm]{$\cC\times\cD$};
    \node(3)[right of= 2,xshift=0.3cm]{$\cD$};
    \node(4)[above of= 2]{$\cX$};
    \draw[a](4)to node[la,swap]{$F$}(1);
    \draw[a](4)to node[la]{$G$}(3);
    \draw[u](4)to node[la,over]{$\exists!\left<F,G\right>$}(2);
    \draw[a](2)to node[la]{$\pi_{\cC}$}(1);
    \draw[a](2)to node[swap,la]{$\pi_{\cD}$}(3);
\end{diagram}

However, the category $\Cat$ has further structure. Indeed, it is a $2$-category with $2$-morphisms the natural transformations between the functors. This $2$-dimensional structure is compatible with the product of categories. More precisely, there is a bijection of natural transformations as depicted below, which is implemented by whiskering with the projection functors.
\begin{center}
    \setcounter{diagram}{\value{definition}}\refstepcounter{diagram}
    \hfill
$\left(\begin{tikzpicture}
    \node(2)[]{$\cC$};
    \node(4)[above of= 2]{$\cX$};
    \draw[a,bend right=40](4)to node(a)[la,swap]{$F$}(2);
    \draw[a,bend left=40](4)to node(b)[la]{$F'$}(2);
    \cell[la,yshift=1pt]{a}{b}{$\alpha$};
  \end{tikzpicture}\ , \
  \begin{tikzpicture}
    \node(2)[]{$\cD$};
    \node(4)[above of= 2]{$\cX$};
    \draw[a,bend right=40](4)to node(a)[la,swap]{$G$}(2);
    \draw[a,bend left=40](4)to node(b)[la]{$G'$}(2);
    \cell[la]{a}{b}{$\beta$};
  \end{tikzpicture}
\right)\quad \leftrightsquigarrow\quad\begin{tikzpicture}
    \node(2)[minimum size=0.7cm]{$\cC\times\cD$};
    \node(4)[minimum size=0.7cm,above of= 2]{$\cX$};
    \draw[a,bend right=45](4)to node(a)[la,swap]{$\left<F,G\right>$}(2);
    \draw[a,bend left=45](4)to
    node(b)[la]{$\left<F',G'\right>$}(2);
    \coordinate (l) at ($(a)-(0,0.1cm)$);
    \coordinate (r) at ($(b)-(0,0.1cm)$);
    \cell[la,xshift=-2pt]{l}{r}{$\left<\alpha,\beta\right>$};
  \end{tikzpicture}$
  \label{diag:bijijiji}
  \hfill{\normalfont(\thediagram)}
\end{center}\refstepcounter{definition}
Observe that the natural transformations $\alpha$ and $\beta$ correspond to functors $\alpha\colon \cX\times \mathbbm{2}\to \cC$ and $\beta\colon\cX\times \mathbbm{2}\to \cD$, where $\mathbbm{2}$ is the category $\{0\to 1\}$. In this light, the bijection (\ref{diag:bijijiji}) of natural transformations can be retrieved by applying the universal property of (\ref{univproduct}) to the functors $\alpha\colon \cX\times \mathbbm{2}\to \cC$ and $\beta\colon\cX\times \mathbbm{2}\to \cD$.

 Taken together, the bijections of (\ref{univproduct}) and (\ref{diag:bijijiji}) assemble into an isomorphism of \emph{categories}
\[\aiinl{\Cat(\cX,\cC\times\cD)}
  {\left<(\pi_{\cC})_{*},(\pi_{\cD})_{*}\right>}
  {\Cat(\cX,\cC)\times\Cat(\cX,\cD)}. \]
This then is the defining feature of a $2$-dimensional limit: there two aspects of the universal property, one for morphisms and one for $2$-morphisms.

In the case of the product above, the indexing category is just a $1$-category. Since $\Cat$ is a $2$-category, one could instead consider indexing diagrams by a $2$-category $I$. In order to define a general $2$-dimensional limit in $\Cat$, we need a \emph{category} of higher morphisms between two $2$-functors. This is the category of $2$-natural transformations and $3$-morphisms, called \emph{modifications}, between them. With these notions, a $2$-limit of a $2$-functor $F\colon I\to \Cat$ can be defined as a pair $(\cL,\lambda)$ of a category $\cL$ and a $2$-natural transformation $\ninl {\Delta \cL}\lambda F$ which are such that post-composition with $\lambda$ gives an isomorphism of categories
\begin{diagram}\label{for:2limit}
  \node {$\aiinl {\Cat(\cX,\cL)}{\lambda_*\circ \Delta}{[I,\Cat](\Delta \cX,F)}$ .};
\end{diagram}

\subsection{\texorpdfstring{$2$}{2}-dimensional conjectures}

A $2$-limit of a general $2$-functor $\ainl IF\cA$ is defined in the same fashion as indicated in (\ref{for:2limit}) above; see \Cref{def:limit}. This notion was first introduced, independently, by Auderset \cite{Auderset} and Borceux-Kelly \cite{BorKel}, and was further developed by Street~\cite{Street1976}, Kelly \cite{Kelly1989,KellyBook} and Lack in \cite{Lack2010}. Motivated by the $1$-categorical case, it is natural to ask whether $2$-limits can be characterised as $2$-dimensional terminal objects in the slice $2$-category of $2$-cones. The appropriate notion of terminality here is that of a \emph{$2$-terminal object} -- an object such that every hom-category to this object is isomorphic to the terminal category $\mathbbm{1}$.

Having seen all other concepts involved, let us introduce now the \emph{slice $2$-category} $\sslice \Delta F$ of $2$-cones over a $2$-functor $F\colon I\to \cA$. This $2$-category has as objects pairs $(X,\mu)$ of an object $X\in \cA$ together with a $2$-cone $\mu\colon \Delta X\Rightarrow F$, and as morphisms those morphisms $\ainl XfY$ of $\cA$ making the $2$-cones commute.
  \begin{diagram}\label{diag:hellotriangle}
  \node(1)[]{$\Delta X$};
  \node(2)[right of= 1]{$\Delta Y$};
  \node(3)[below of= 2]{$F$};
  \draw[n](1)to node[la]{$\Delta f$}(2);
  \draw[n](1)to node[swap,la]{$\mu$}(3);
  \draw[n](2)to node[la]{$\nu$}(3);
\end{diagram}
The $2$-morphisms are given by $2$-morphisms in $\cA$ which satisfy a certain whiskering identity.

This slice $2$-category seems appropriate to our conjecture since the $1$-dimensional aspect of the universal property of a $2$-terminal object in there is exactly the same as the $1$-dimensional aspect of the universal property of a $2$-limit. In the special case of~$\Cat$, by generalising the argument we have seen for products, the $1$-dimensional aspect of the universal property of a $2$-limit in \Cat{} suffices to reconstruct its $2$-dimensional aspect. This holds more broadly in every $2$-category admitting tensors by $\mathbbm{2}$, as demonstrated in \Cref{prop:tensor}. Given this, we conjecture:

\begin{conjecture} \label{conj1}
  Let $I$ and $\cA$ be $2$-categories, and let $\ainl IF\cA$ be a
  $2$-functor. Let $L\in\cA$ be an object and
  $\ninl {\Delta L}\lambda F$ be a $2$-natural transformation. The
  following two statements are equivalent:
  \begin{rome}[topsep=3pt]
  \item The pair $(L,\lambda)$ is a $2$-limit of the functor $F$.
  \item The pair $(L,\lambda)$ is a $2$-terminal object in the slice
    $2$-category $\sslice \Delta F$ of $2$-cones over~$F$.
  \end{rome}
\end{conjecture}

Although it gives the authors no great pleasure to mislead the reader so, such a conjecture is \emph{false}. While a $2$-limit is always a $2$-terminal object in the slice $2$-category of $2$-cones (see \Cref{pos:2l2ts}), the converse is not necessarily true (see \Cref{ce:2tsnot2l}). The reason for this failure is that the $2$-dimensional aspect of the universal property of a $2$-terminal object in the slice $2$-category is weaker than the $2$-dimensional aspect of the universal property of a $2$-limit. This manifests in, among other things, the inability of the slice $2$-category to detect enough modifications between two $2$-cones with the same summit.

This is not, however, the last word for this conjecture. The theory of $2$-categories affords us more room than that of categories to coherently weaken various notions. Instead of considering a $2$-category of $2$-cones where the morphisms render triangles like that of~(\ref{diag:hellotriangle}) commutative, we are free to ask that the data of a morphism comprises also the data of a general, or perhaps invertible, modification filling that triangle. This leads to the notions of \emph{lax-slice} and \emph{pseudo-slice} $2$-categories of $2$-cones, respectively. Unlike the original, \emph{strict}, slice $2$-category given above, the lax- and pseudo-slice $2$-categories detect all, or more, modifications between $2$-cones. With this in mind, we might conjecture:

\begin{conjecture} \label{conj2}
  Let $I$ and $\cA$ be $2$-categories, and let $\ainl IF\cA$ be a $2$-functor. Let $L\in\cA$ be an object and $\ninl {\Delta L}\lambda F$ be a $2$-natural transformation. The
  following two statements are equivalent:
  \begin{rome}[topsep=3pt]
  \item The pair $(L,\lambda)$ is a $2$-limit of the $2$-functor $F$.
  \item The pair $(L,\lambda)$ is a $2$-terminal object in the lax-slice (or pseudo-slice) $2$-category of $2$-cones over $F$.
  \end{rome}
\end{conjecture}

Unfortunately, this too is incorrect. The failure here is twofold: the pseudo-slice may still fail to detect enough modifications, as before, while simultaneously allowing too many new morphisms to appear. Similar issues plague the lax-slice, and as we will see in \Cref{sec:laxpsdslice}, $2$-terminal objects in either are generally unrelated to $2$-limits.

At this point, it is natural to ask whether the failure of \Cref{conj1,conj2} has something to do with the rigidity of the notion of $2$-limits. We might, for instance, ask that our $2$-cones have naturality triangles filled by a general $2$-morphism, or perhaps any invertible $2$-morphism, not just the identity. This leads us to consider \emph{lax-limits} and \emph{pseudo-limits}. One might imagine that these limits have special relationships with the lax- and pseudo-slices of matching cones, respectively. Specifically, one might hope that the peculiarities of these weaker notions of $2$-dimensional limit conspire somehow to support the following conjectures.

\begin{conjecture} \label{conj4}
    Let $I$ and $\cA$ be $2$-categories, and let $\ainl IF\cA$ be a $2$-functor. Let $L\in\cA$ be an object and $\ninl {\Delta L}\lambda F$ be a pseudo-natural (resp.~lax-natural) transformation. The following two statements are equivalent:
  \begin{rome}
  \item The pair $(L,\lambda)$ is a pseudo-limit (resp.~lax-limit) of the functor $F$.
  \item The pair $(L,\lambda)$ is a $2$-terminal object in the strict-/pseudo-/lax-slice $2$-category of pseudo-cones (resp.~lax-cones) over $F$.
  \end{rome}
\end{conjecture}

As in the case of $2$-limits, pseudo- and lax-limits are in particular $2$-terminal objects in the strict-slice of appropriate cones (see \Cref{rem:pllltss}). However, all other implications are generally false, as established in \Cref{sec:psdlaxlimit}.

We might then ask whether the failure of \Cref{conj1,conj2,conj4} has had, all along, something to do with the rigidity of universal properties expressed by isomorphisms of categories. One might, on this view, hope to generate analogous and valid conjectures by weakening these isomorphisms to equivalences of categories -- conjectures concerning \emph{bi}-type limits and \emph{bi}-terminal objects. However, as discussed in \Cref{sec:bilimits}, even these do not hold.

Finally, one might wonder about the case of weighted $2$-limits, which is a well-established notion for limits in enriched category theory in the literature. The theory of weighted limits was developed by Auderset \cite{Auderset}, Street \cite{Street1976}, and Kelly \cite{Kelly1989} in the case of $2$-categories, and by Borceux-Kelly \cite{BorKel} as well as Kelly \cite[Chapter~3]{KellyBook} in the case of general enriched categories. However, as conical $2$-limits, pseudo-limits, and lax-limits are special cases of such weighted $2$-limits as noted in \cite[\S 3 and \S 5]{Kelly1989}, we can see that the analogues of \cref{conj1,conj2,conj4} for weighted $2$-limits must in general fail too.

\subsection{Outline}\label{sec:summary}

The structure of the paper is as follows.
In \cref{sec:2limits}, we introduce the notions of $2$-limits and strict-slices of $2$-cones. We prove that a $2$-limit is always $2$-terminal in the strict-slice of $2$-cones, but we provide a counter-example demonstrating that the converse fails in general. However, for the converse to hold, it is sufficient for the ambient $2$-category to admit tensors by $\mathbbm 2$ -- as is the case of the $2$-category \Cat{}. In \cref{sec:laxpsdslice}, we turn our attention to the larger $2$-categories of pseudo- and lax-slices of $2$-cones. We provide counter-examples demonstrating that $2$-limits are in fact unrelated to $2$-terminal objects in these -- neither notion generally implies the other. In \cref{sec:psdlaxlimit}, we introduce pseudo- and lax-limits, and investigate their relationships with $2$-terminal objects in the different slices. Finally, in \cref{sec:bilimits}, we address the case of bi-type limits. We show that these are in particular always bi-terminal in the pseudo-slice of appropriate cones, and then adapt the results we have for the $2$-type cases to the bi-type cases.

All of the counter-examples presented in this paper, with the exception of \Cref{ce:2tssnotll}, are indexed by finite ($1$-)categories: $\mathbbm 2$ and the pullback shape specifically. These counter-examples were generally constructed to exhibit certain data, for example a modification that is not detected by the strict-slice or a $2$-morphism which gives an errant morphism in the lax-slice. The resulting diagram shapes were inessential to this process, and there are certainly many more counter-examples yet.

In the tables below, we summarise our counter-examples and reductions for \cref{conj1,conj2,conj4} -- see \Cref{sec:bilimits} for the matching tables for the bi-type conjectures. Only results marked with a \valid{} are true, everything else establishes a counter-example.
Note that the objects in the slices considered vary by the column: the type of objects should match the type of the limit cone.

\begin{table}[h!]
  \centering
  \caption{$2$-type limits which are not $2$-terminal}
  \begin{tabular}{*{3}{c|}c}
    $\to$ implies $2$-terminal in $\downarrow$ & $2$-limit & Pseudo-limit & Lax-limit\\\hline
    Strict-slice & \valid{} \cref{pos:2l2ts} & \valid{} \cref{rem:pllltss} & \valid{} \cref{rem:pllltss} \\
    Lax-slice & \cref{ce:2lnot2tl} & \Cref{red:plnot2tl} & \cref{ce:llnot2tls} \\
     Pseudo-slice & \Cref{rem:2lnottp} & \Cref{red:llplnot2tp} & \cref{red:llplnot2tp}
  \end{tabular}
\end{table}

\begin{table}[h!]
  \centering
  \caption{$2$-terminal objects which are not $2$-type limits}
  \begin{tabular}{*{3}{c|}c}
    $2$-terminal in $\downarrow$ implies $\to$ & $2$-limit & Pseudo-limit & Lax-limit \\\hline
    Strict-slice & \cref{ce:2tsnot2l} & \cref{rem:tsnotllpl} & \cref{ce:2tssnotll} \\
    Lax-slice & \cref{ce:tlsnot2l} & \cref{rem:pseudobad} & \cref{ce:2tlnotll} \\
    Pseudo-slice & \cref{rem:2tpnot2l} & \cref{ce:2tpnotpl} & \cref{ce:2tpnotpl}
  \end{tabular}
\end{table}

In consulting these tables, some readers might be confused that the adjectives ``pseudo'' and ``lax'' do not appear in the same order in the rows as they appear in the columns. We should be careful to consider that these adjectives play very different roles when attached to the various slice $2$-categories as they do when attached to a notion of limit. Adding these adjectives to the labels of the rows changes the \emph{morphisms} of the slices, while adding these adjectives to the labels of the columns changes the \emph{type of the cones}, i.e.~it changes the \emph{objects} of the slices. The unexpected ordering of the columns of the tables has been chosen to be this way, since all counter-examples for pseudo-slices are reductions of counter-examples for lax-slices.

\subsection{Positive results for characterisations of 2-dimensional limits}

We may always view a $2$-category as a horizontal double category with only trivial vertical morphisms, and in the double categorical setting we are now afforded a stronger notion of terminality. In this broader context, Grandis-Paré show in \cite{GraPar1999,GraPar2019} that (weighted) $2$-limits of a $2$-functor~$F$ are equivalently \emph{double terminal} objects in the double category of (weighted) cones over the horizontal double functor induced by $F$; see also \cite[\S 5.6]{Grandis}. Similar work in this direction is done by Verity in his thesis \cite{Verity}. 
With this proliferation of positive results, it is surprising that the failures of \cref{conj1,conj2,conj4} are not documented in the literature. Grandis-Paré are certainly aware of such a failure as they write the following in their recent paper \cite{GraPar2019}:
\begin{quote}
        On the other hand, there seems to be no natural way of expressing the $2$-dimensional universal property of weighted (strict or pseudo) limits by terminality in a $2$-category.
\end{quote}
Unfortunately however, Grandis-Par\'e do not record their formulation of the ``natural way'' nor whatever obstacles they encountered. We feel that \cref{conj1,conj2,conj4} express a natural expectation of the relationship between $2$-limits and $2$-terminal objects in a $2$-category, and we hope that our counter-examples illustrate clearly the failure of all such conjectures.

Closer examination of these counter-examples reveals the need to capture additional information not present in the slice $2$-category of cones. The double categorical approach of Grandis, Par\'e, and Verity certainly suffices for this task, but in our paper \cite{tslil} we give a purely $2$-categorical characterisation of $2$-limits by constructing two different slice $2$-categories of cones which have the joint property that objects which are simultaneously $2$-terminal in both correspond precisely to $2$-limits. One of these slice $2$-categories is predictably the slice $2$-category of cones, but in fact the other slice $2$-category alone succeeds in precisely characterising $2$-limits through bi-initial objects of a specific form. This second slice $2$-category, however, is a shifted version of the usual slice $2$-category of cones: its \emph{objects} are modifications between cones. An advantage of this approach will be highlighted in forthcoming work by the second author, where a notion of $(\infty,2)$-limits can then be defined in a fully $(\infty,2)$-categorical language without requiring the development of the accompany theory of double $(\infty,1)$-categories.

The counter-examples in this paper are indicative of a larger failure in the extension of $1$-categorical theorems to the setting of $2$-category theory. More generally, the existence and characterisations of bi-limits may be viewed as an instance of the corresponding problems for \emph{bi-representations} of general pseudo-presheaves, and it is here that the analogy breaks down: while a representation for a presheaf corresponds to an initial object in the category of elements, the data of a bi-representation for a pseudo-presheaf is not wholly captured by a bi-initial object in the $2$-category of elements.

At the level of $2$-dimensional representations, in \cite{tslil} we weaken the strict setting to that of pseudo-functors and pseudo-natural transformations and generalise the results of Grandis, Par\'e, and Verity to the case of bi-representations. In particular we give a double categorical characterisation of bi-re\-pre\-sen\-ta\-tions of pseudo-presheaves in terms of double bi-initial objects in the double category of elements. Furthermore, we succeed in providing a purely $2$-categorical characterisation of bi-representations in terms of objects which are simultaneously bi-initial in the familiar $2$-category of elements and in a new $2$-category of \emph{morphisms}. In fact, we are able to demonstrate that bi-representations can actually be characterised as bi-initial objects of a specific form in the $2$-category of morphisms alone. These results are the content of \cite[Theorem 6.8]{tslil}. The counter-examples of this paper establish the necessity of the presence of both $2$-categories in the theorems there, as bi-limits are bi-representations. As a corollary of these theorems we obtain a purely $2$-categorical characterisation of weighted bi-limits in \cite[Theorem 7.19]{tslil}. 

Finally, the positive results of \Cref{prop:tensor,prop:bitensor} are special cases of more general results for bi-representations: \cite[Theorem 6.14]{tslil} shows that in the presence of tensors by $\mathbbm 2$, if the pseudo-presheaf preserves such tensors, then bi-representations are precisely bi-initial objects in the $2$-category of elements.

\subsection{Acknowledgements}

Both authors are indebted to Emily Riehl for her close readings of and thoughtful inputs on several early drafts of this paper. In addition, both authors are grateful to J\'er\^ome Scherer for his careful input on an early draft. Finally, both authors also wish to extend their gratitude to Alexander Campbell and Emily Riehl for their enthusiasm for what became \Cref{ce:2tlnotll}, which provided the impetus to write this paper.

This work was realised while both authors were at the Mathematical Sciences Research Institute in Berkeley, California, during the Spring 2020 semester. The first-named author benefited from support by the National Science Foundation under Grant No. DMS-1440140, while at residence in MSRI. The second-named author was supported by the Swiss National Science Foundation under the project P1ELP2\_188039. The first-named author was additionally supported by the National Science Foundation grant DMS-1652600, as well as the JHU Catalyst Grant.

\section{\texorpdfstring{$2$}{2}-limits do not correspond to \texorpdfstring{$2$}{2}-terminal objects in the strict-slice} \label{sec:2limits}

In this section, we start by comparing $2$-limits with $2$-terminal objects in the strict-slice $2$-category of $2$-cones. After introducing all the terms involved, we show that a $2$-limit is in particular a $2$-terminal object in the strict-slice, but we provide a counter-example for the other implication. However, when the ambient $2$-category admits tensors by $\mathbbm{2}$, such as is the case of $\Cat$, these two notions do coincide.

A $2$-category has not only the structure of a category, with objects and morphisms, but additionally has $2$-morphisms between parallel morphisms. These $2$-morphisms may be composed both vertically, along a common morphism boundary, and horizontally, along a common object boundary. To differentiate these two types of composition operations on $2$-morphisms, we write $*$ for horizontal composition and use juxtaposition to denote vertical composition. A $2$-functor between $2$-categories comprises maps of objects, morphisms, and $2$-morphisms which preserves all $2$-categorical structures strictly. There are also notions of $2$- and $3$-morphisms between $2$-categories, which we introduce now.

\begin{definition} \label{def:2nat} Let $F,G\colon I\to \cA$ be
  $2$-functors. A \textbf{$2$-natural transformation}
  $\mu\colon F\Rightarrow G$ comprises the data of a morphism
  $\mu_i\colon Fi\to Gi$ of $\cA$ for each $i\in I$, which must
  satisfy
  \begin{enumerate}
  \item for all morphisms $\ainl ifj$ of $I$, we have
    $(Gf)\mu_i=\mu_j(Ff)$, and
  \item for all $2$-morphisms $\ninl f\alpha g$ of $I$, we have
    $G\alpha *\mu_i=\mu_j*F\alpha$.
  \end{enumerate}
  \begin{diagram*}[node distance=1.7cm]
    \node(1)[]{$Fi$};
    \node(2)[right of= 1]{$Gi$};
    \node(3)[right of= 2]{$Gj$};
    \draw[a](1)to node[la]{$\mu_{i}$}(2);
    \draw[a,bend left=40](2)to node(t1)[la]{$Gf$}(3);
    \draw[a,bend right=40](2)to node(b1)[la,swap]{$Gg$}(3);

    \coordinate (t1') at ($(t1)-(0.1cm,0)$);
    \coordinate (b1') at ($(b1)-(0.1cm,0)$);
    \cell[la,xshift=1pt]{t1'}{b1'}{$G\alpha$};

    \node(1)[right of=3]{$Fi$};
    \node at ($(3)!0.5!(1)$) {$=$};
    \node(2)[right of= 1]{$Fj$};
    \node(3)[right of= 2]{$Gj$};
    \draw[a](2)to node[la]{$\mu_{j}$}(3);
    \draw[a,bend left=40](1)to node(t1)[la]{$Ff$}(2);
    \draw[a,bend right=40](1)to node(b1)[la,swap]{$Fg$}(2);

    \coordinate (t1') at ($(t1)-(0.1cm,0)$);
    \coordinate (b1') at ($(b1)-(0.1cm,0)$);
    \cell[la,xshift=1pt]{t1'}{b1'}{$F\alpha$};
  \end{diagram*}
\end{definition}

\begin{definition} \label{def:strictmod} Let $\ainl {I}{F,G}{\cA}$ be
  $2$-functors and let $\mu,\nu\colon F\Rightarrow G$ be $2$-natural
  transformations. A \textbf{modification} $\minl{\mu}{\varphi}{\nu}$
  comprises the data of a $2$-morphism
  $\varphi_i\colon \mu_i\Rightarrow \nu_i$ for each $i\in I$, which
  satisfy $Gf*\varphi_i=\varphi_j*Ff$, for all morphisms
  $f\colon i\to j$ of $I$.
  \begin{diagram*}[node distance=1.7cm]
      \node(1)[]{$Fi$};
      \node(2)[right of= 1]{$Gi$};
      \node(3)[right of= 2]{$Gj$};
      \draw[a](2)to node[la]{$Gf$}(3);
      \draw[a,bend left=40](1)to node(t1)[la]{$\mu_{i}$}(2);
      \draw[a,bend right=40](1)to node(b1)[la,swap]{$\nu_{i}$}(2);

      \coordinate (t1') at ($(t1)-(0.1cm,0)$);
      \coordinate (b1') at ($(b1)-(0.1cm,0)$);
      \cell[la,xshift=1pt]{t1'}{b1'}{$\varphi_{i}$};

      \node(1)[right of=3]{$Fi$};
      \node at ($(3)!0.5!(1)$) {$=$};
      \node(2)[right of= 1]{$Fj$};
      \node(3)[right of= 2]{$Gj$};
      \draw[a](1)to node[la]{$Ff$}(2);
      \draw[a,bend left=40](2)to node(t1)[la]{$\mu_{j}$}(3);
      \draw[a,bend right=40](2)to node(b1)[la,swap]{$\nu_{j}$}(3);

      \coordinate (t1') at ($(t1)-(0.1cm,0)$);
      \coordinate (b1') at ($(b1)-(0.1cm,0)$);
      \cell[la,xshift=1pt]{t1'}{b1'}{$\varphi_{j}$};
  \end{diagram*}
\end{definition}

With these definitions, $2$-functors, $2$-natural transformations, and modifications assemble into a $2$-category.

\begin{notation}
  Let $I$ and $\cA$ be $2$-categories. We denote by $[I,\cA]$ the
  $2$-category of $2$-functors $I\to \cA$, $2$-natural transformations
  between them, and modifications.
\end{notation}

We are now ready to define $2$-dimensional limits.

\begin{definition} \label{def:limit} Let $I$ and $\cA$ be
  $2$-categories, and let $F\colon I\to \cA$ be a $2$-functor. A
  \textbf{$2$-limit} of $F$ comprises the data of a object $L\in \cA$
  together with a $2$-natural transformation $\lambda\colon \Delta L\Rightarrow F$, which are such that, for each
  object $X\in\cA$, the functor
  \[ \ainl{\cA(X,L)}{\lambda_*\circ\Delta}{[I,\cA](\Delta X, F)} \]
  given by post-composition with $\lambda$ is an isomorphism of
  categories.
\end{definition}

In what follows, we call a $2$-natural transformation $\Delta X\Rightarrow F$ from a constant functor a \emph{$2$-cone over $F$}.

\begin{remark} \label{rem:univprop}
  There are two aspects of the universal property of a $2$-limit, which arise from the isomorphism of categories
  $\ainl{\cA(X,L)}{\lambda_*\circ\Delta}{[I,\cA](\Delta X, F)}$
  at the level of objects and at the level of morphisms. We reformulate this more explicitly as follows. For every $X\in \cA$,
  \begin{enumerate}
      \item  for every $2$-cone $\mu\colon \Delta X\Rightarrow F$, there is a unique morphism $f_{\mu}\colon X\to L$ in $\cA$ such that $\lambda\Delta f_{\mu}=\mu$,
      \item for every modification $\minl{\mu}{\varphi}{\nu}$ between $2$-cones $\mu,\nu\colon \Delta X\Rightarrow F$, there is a unique $2$-morphism $\alpha\colon f_{\mu}\Rightarrow f_{\nu}$ in $\cA$ such that $\lambda*\Delta\alpha=\varphi$.
      \begin{diagram}[node distance=1.7cm]\label{diag:2limituni}
      \node(4)[]{$\Delta X$};
      \node(5)[right of= 4]{$\Delta L$};
      \node(6)[right of= 5]{$F$};
      \node(1)[right of=6]{$\Delta X$};
      \node(2)[right of= 1]{$F$};
      \draw[n,bend left=40](1)to node(m)[la,above]{$\mu$}(2);
      \draw[n,bend right=40](1)to node(n)[la,below]{$\nu$}(2);
      \draw[n,bend left=40](4)to node(t)[la]{$\Delta f_{\mu}$}(5);
      \draw[n,bend right=40](4)to node(b)[swap,la]{$\Delta f_{\nu}$}(5);
      \draw[n](5)to node[la]{$\lambda$}(6);

      \cell[la,xshift=1pt][t]{m}{n}{$\varphi$};

      \coordinate (n) at ($(t)-(0.2cm,0)$);
      \coordinate (m) at ($(b)-(0.2cm,0)$);
      \cell[la,xshift=1pt][t]{n}{m}{$\Delta\alpha$};

      \node at ($(6)!0.5!(1)$) {$=$};
    \end{diagram}
  \end{enumerate}
\end{remark}

We now define strict-slice $2$-categories of $2$-cones and $2$-terminal objects.

\begin{definition} \label{def:strictslice}
Let $F\colon I\to \cA$ be a $2$-functor. The \textbf{strict-slice} $\sslice \Delta F$ of $2$-cones over
  $F$ is defined to be the following pullback in the ($1$-)category of
  $2$-categories and $2$-functors.
  \begin{diagram*}
    \node(1)[]{$\sslice \Delta F$};
    \node(2)[below of= 1]{$\cA\times \mathbbm{1}$};
    \node(3)[right of= 2, xshift=1.5cm]{$[I,\cA]\times [I,\cA]$};
    \node(4)[above of= 3]{$[\mathbbm{2},[I,\cA]]$};
    \draw[a](1)to (2);
    \draw[a](2)to node[swap,la]{$(\Delta,F)$}(3);
    \draw[a](1)to (4);
    \draw[a](4) to node[la]{$(s,t)$}(3);
    \pullback{1};
  \end{diagram*}
  This $2$-category $\sslice \Delta F$ is given by the following data:
  \begin{rome}
  \item an object in $\sslice \Delta F$ is a pair $(X,\mu)$ of an
    object $X\in \cA$ together with a $2$-natural transformation
    $\mu\colon \Delta X\Rightarrow F$,
  \item a morphism $f\colon (X,\mu)\to (Y,\nu)$ consists of
    a morphism $f\colon X\to Y$ in $\cA$ such that $\nu\Delta f=\mu$,
  \item a $2$-morphism
    $\alpha\colon f\Rightarrow g$ between morphisms $f,g\colon (X,\mu)\to (Y,\nu)$ is a
    $2$-morphism $\alpha\colon f\Rightarrow g$ in $\cA$ such that $\lambda*\Delta\alpha=\id_{\mu}$.
      \begin{diagram}[node distance=1.7cm]\label{diag:2termuni}
      \node(4)[]{$\Delta X$};
      \node(5)[right of= 4]{$\Delta L$};
      \node(6)[right of= 5]{$F$};
      \node(1)[right of=6]{$\Delta X$};
      \node(2)[right of= 1]{$F$};
      \draw[n](1)to node(m)[la]{$\mu$}(2);
      \draw[n,bend left=40](4)to node(t)[la]{$\Delta f$}(5);
      \draw[n,bend right=40](4)to node(b)[swap,la]{$\Delta g$}(5);
      \draw[n](5)to node[la]{$\lambda$}(6);

      \coordinate (n) at ($(t)-(0.2cm,0)$);
      \coordinate (m) at ($(b)-(0.2cm,0)$);
      \cell[la,xshift=1pt][t]{n}{m}{$\Delta\alpha$};

      \node at ($(6)!0.5!(1)$) {$=$};
    \end{diagram}
  \end{rome}
\end{definition}

\begin{definition} \label{def:2terminal}
  Let $\cA$ be a $2$-category. An object $L\in \cA$ is
  \textbf{$2$-terminal} if for all $X\in \cA$ there is an isomorphism
  of categories $\cA(X,L)\cong \mathbbm{1}$.
\end{definition}

As for $2$-limits, there are also two aspects of the universal property of a $2$-terminal object. Since we are interested here by $2$-terminal objects in a strict-slice $2$-category of $2$-cones, we give a more explicit description of their universal property. We will then compare this description with the universal property of $2$-limits (c.f.~\Cref{rem:univprop}).

\begin{remark} \label{rem:univ2terminal}
  Given a $2$-functor $F\colon I\to \cA$, we describe the two aspects of the universal property of a $2$-terminal object in the strict-slice $\sslice \Delta F$. Such a $2$-terminal object comprises the data of an object $L\in \cA$ together with a $2$-cone $\lambda\colon \Delta L\Rightarrow F$ which satisfy the following:
  \begin{enumerate}
      \item for every $X\in \cA$ and every $2$-cone $\mu\colon \Delta X\Rightarrow F$, there is a unique morphism $f_{\mu}\colon X\to L$ in $\cA$ such that $\lambda \Delta f_{\mu}=\mu$,
      \item for every $X\in \cA$ and every $2$-cone $\mu\colon \Delta X\Rightarrow F$, the unique $2$-morphism $f_{\mu}\Rightarrow f_{\mu}$ in $\sslice \Delta F$ is the identity~$\id_{f_{\mu}}$.
  \end{enumerate}
\end{remark}

In particular, we can see that the $2$-dimensional aspect above seems somehow degenerate in comparison to (2) of \Cref{rem:univprop}. However, the $1$-dimensional aspect is the same as the one expressed in \Cref{rem:univprop} (1). This gives the following result.

\begin{proposition}\label{pos:2l2ts}
  Let $I$ and $\cA$ be $2$-categories, and let $\ainl IF\cA$ be a
  $2$-functor. If $(L,\lambda\colon \Delta L\Rightarrow F)$ is a $2$-limit
  of $F$, then $(L,\lambda)$ is $2$-terminal in the strict-slice
  $\sslice {\Delta} F$ of $2$-cones over $F$.
\end{proposition}

\begin{proof}
  By \Cref{rem:univprop} (1) and \Cref{rem:univ2terminal} (1), we observe that the $1$-dimensional aspects of the universal property of a $2$-limit and of a $2$-terminal object in the strict-slice coincide. Both say that, for every $X\in \cA$ and every $2$-cone $\mu\colon \Delta X\Rightarrow F$, there exists a unique morphism $f_{\mu}\colon X\to L$ in $\cA$ such that $\lambda\Delta f_{\mu}=\mu$.

  It remains to show (2) of \Cref{rem:univ2terminal}, that is, that the unique $2$-morphism $f_{\mu}\Rightarrow f_{\mu}$ in $\sslice \Delta F$ is the identity $\id_{f_{\mu}}$. Any $2$-morphism $\ninl{f_{\mu}}{\alpha}{f_{\mu}}$ in $\sslice \Delta F$ must satisfy $\lambda*\Delta\alpha=\id_{\lambda\Delta f}$ by \Cref{def:strictslice} (iii). In particular, we also have $\lambda*\Delta\id_f=\id_{\lambda\Delta f}=\lambda*\Delta\alpha$. By the uniqueness in \Cref{rem:univprop} (2), it follows that $\alpha=\id_{f}$.
\end{proof}

However, it is not true that every $2$-terminal object in the strict-slice of $2$-cones is a $2$-limit. One reason for this is that the strict-slice only sees the identity modifications between $2$-cones (compare (\ref{diag:2termuni}) with (\ref{diag:2limituni})). With this in mind, to illustrate this failure we give an example of a modification between two $2$-cones which does not arise from a $2$-morphism.

\begin{ce}\label{ce:2tsnot2l}
  Let $I$ be the pullback shape diagram $\{ \bullet \longrightarrow\bullet \longleftarrow \bullet\}$. Let $\cA$ be the $2$-category generated by the data
  \begin{diagram*}
    \node(1)[]{$X$};
    \node(2)[right of= 1]{$L$};
    \node(3)[below right of= 1]{$L$};
    \node(4)[below right of= 2]{$B$};
    \node(5)[below of= 3]{$C$};
    \node(6)[below of= 4]{$A$};
    \draw[a](1)to node[la]{$g$}(2);
    \draw[a](1)to node[swap,la]{$f$}(3);
    \draw[a](2)to node[la]{$\lambda_{0}$}(4);
    \draw[a](3)to node[la,over]{$\lambda_{0}$}(4);
    \draw[a](3)to node[swap,la]{$\lambda_{1}$}(5);
    \draw[a](4)to node[la]{$b$}(6);
    \draw[a](5)to node[swap,la]{$c$}(6);
    \cell[la,near start]{3}{2}{$\gamma_{0}$};
    \node(7)[right of= 2,xshift=2cm]{$X$};
    \node(8)[below of= 7]{$L$};
    \node(9)[below right of= 7]{$L$};
    \node(10)[right of= 9]{$B$};
    \node(11)[below of= 9]{$C$};
    \node(12)[below of= 10]{$A$};
    \draw[a](7)to node[swap,la]{$f$}(8);
    \draw[a](7)to node[la]{$g$}(9);
    \draw[a](8)to node[swap,la]{$\lambda_{1}$}(11);
    \draw[a](9)to node[la,over]{$\lambda_{1}$}(11);
    \draw[a](9)to node[la]{$\lambda_{0}$}(10);
    \draw[a](10)to node[la]{$b$}(12);
    \draw[a](11)to node[swap,la]{$c$}(12);
    \cell[la]{8}{9}{$\gamma_{1}$};
    \node at ($(7)!0.5!(6)$) {$=$};
  \end{diagram*}
  subject to the relations $b\lambda_0=c\lambda_1$ and $b*\gamma_0=c*\gamma_1$.
  Take $F\colon I\to \cA$ to be the diagram
  \begin{center}
  \begin{tikzpicture}[node distance=1.3cm, baseline=(0.base)]
      \node(1)[]{$B$};
      \node(0)[below of= 1]{$A$};
      \node(2)[left of= 0]{$C$};
      \draw[a] (2) to node[swap,la]{$c$} (0);
      \draw[a] (1) to node[la]{$b$} (0);
    \end{tikzpicture} .
    \end{center}
\end{ce}

\begin{claim}
The object $(L,\ninl{\Delta L}{\lambda}{F})$ is $2$-terminal in the strict-slice $\sslice \Delta F$ of $2$-cones over~$F$, but the functor
\[ \ainl{\cA(X,L)}{\lambda_*\circ\Delta}{[I,\cA](\Delta X,F)} \] given
by post-composition with $\lambda$ is not surjective on morphisms,
i.e.~$(L,\lambda)$ is not a $2$-limit of $F$.
\end{claim}

\begin{proof}
  The objects of the strict-slice $\sslice \Delta F$ are given by the
  $2$-cones over $F$:
\[ (L,\lambda),\quad (X,\lambda* f), \ \  \text{and} \ \ (X,\lambda* g).
\]
Each of these objects admits precisely one morphism to $(L,\lambda)$
in $\sslice \Delta F$ given by
\begin{align*}
    \id_L\colon &(L,\lambda)\to (L,\lambda) \\
    f\colon &(X,\lambda* f)\to (L,\lambda) \\
    g\colon &(X,\lambda* g)\to (L,\lambda).
\end{align*}
There are no non-trivial $2$-morphisms to $(L,\lambda)$ in
$\sslice \Delta F$, since there are no non-trivial $2$-morphisms between
$X$ and $L$ in $\cA$. This proves that $(L,\lambda)$ is
$2$-terminal in $\sslice \Delta F$.

However, the $2$-morphisms $\gamma_0$ and $\gamma_1$ give the data of
a modification $\minl{\lambda*f}\gamma{\lambda*g}$, i.e.~a morphism in
$[I,\cA](\Delta X, F)$. But there is no $2$-morphism between $f$ and $g$
in $\cA$ that maps to $\gamma$ via
$\ainl{\cA(X,L)}{\lambda_*\circ\Delta}{[I,\cA](\Delta X, F)}$. Hence
$(L,\lambda)$ is not the $2$-limit of $F$.
\end{proof}

A $2$-terminal object in the strict-slice is, however, a $2$-limit when the $2$-category $\cA$ admits tensors by the category $\mathbbm{2}=\{0\to 1\}$. Indeed, it follows from this condition that the $1$-dimensional aspect of the universal property of a $2$-limit implies the $2$-dimensional one (see \cite[\S 4]{Kelly1989}). A $2$-category \cA is said to \emph{admit tensors by a category \cC} when, for each object $X\in \cA$, there exists an object $X\otimes \cC\in \cA$ together with isomorphisms of categories
\[
\cA(X\otimes \cC, Y) \iso \Cat(\cC,\cA(X,Y))\ ,
\]
$2$-natural in $X,Y\in \cA$. In particular, this implies that there is a bijection between morphisms $X\otimes \cC\to Y$ in $\cA$ and functors $\cC\to \cA(X,Y)$.

\begin{proposition} \label{prop:tensor}
  Suppose $\cA$ is a $2$-category that admits tensors by $\mathbbm{2}$, and let $F\colon I\to \cA$ be a $2$-functor. Then an object $(L,\lambda\colon \Delta L\Rightarrow F)$ is a $2$-terminal object in the strict-slice $\sslice \Delta F$ if and only if it is a $2$-limit of $F$.
\end{proposition}

\begin{proof}
  We already saw one of the implications in \Cref{pos:2l2ts}. Let us prove the other.

  Suppose that $(L,\lambda)$ is a $2$-terminal object in the strict-slice $\sslice \Delta F$. We show that $(L,\lambda)$ satisfies the two conditions (1) and (2) of \Cref{rem:univprop}, expressing the two aspects of the universal property of a $2$-limit. It is clear that (1) holds since it is the same condition as the one expressing the $1$-dimensional aspect of the universal property of a $2$-terminal object, as in \Cref{rem:univ2terminal} (1). It remains to show (2).

  Before proceeding, let us examine the effect of admitting tensors by $\mathbbm 2$. Let $X\in \cA$. The universal property of tensoring by $\mathbbm{2}$ in $\cA$ gives a canonical bijection between morphisms $X\otimes \mathbbm{2}\to L$ in $\cA$ and functors $\mathbbm{2}\to \cA(X,L)$, but these functors coincide with $2$-morphisms from $X$ to $L$. The $2$-category $[I,\cA]$ is also tensored by $\mathbbm{2}$, since $\cA$ is, and the tensor is given object-wise \cite[\S 3.3]{KellyBook}. In particular, $\Delta(X\otimes \mathbbm{2})=\Delta X\otimes \mathbbm{2}$ as constant functors, by the object-wise definition of tensoring by $\mathbbm 2$. As $[I,\cA]$ admits tensors by $\mathbbm 2$, we have a canonical bijection between $2$-cones $\Delta X\otimes \mathbbm{2}\Rightarrow F$ and functors $\mathbbm{2}\to [I,\cA](\Delta X,F)$, which in turn coincide with modifications between $\Delta X$ and $F$.

  By \Cref{rem:univ2terminal} (1), for every $2$-cone $\varphi\colon\Delta X\otimes \mathbbm{2}\Rightarrow F$, there exists a unique morphism $\alpha\colon X\otimes \mathbbm{2}\to L$ in $\cA$ such that $\lambda\Delta \alpha=\varphi$. Using the above, we can reformulate this statement as follows: for every modification $\varphi$ between $2$-cones $\Delta X\Rightarrow F$, there is a unique $2$-morphism $\alpha$ between morphisms $X\to L$ such that $\lambda\Delta \alpha=\varphi$. But this is exactly (2) of \Cref{rem:univprop}.
\end{proof}

The category $\Cat$ of categories and functors is cartesian closed. Therefore, it is enriched over itself and so is, in particular, tensored over $\Cat$. In other words, the $2$-category $\Cat$ of categories, functors, and natural transformations admits tensors by all categories, and these tensors are given by cartesian products. In particular, \Cref{prop:tensor} yields the following result.

\begin{corollary}\label{cor:worksforcat}
Let $F\colon I\to \Cat$ be a $2$-functor into $\Cat$. A pair $(\cL,\lambda\colon \Delta \cL\Rightarrow F)$ is a $2$-terminal in the strict-slice $\sslice \Delta F$ of $2$-cones over~$F$ if and only if it is a $2$-limit of $F$.
\end{corollary}

\section{\texorpdfstring{$2$}{2}-terminal objects in lax- and pseudo-slices are not related to \texorpdfstring{$2$}{2}-limits} \label{sec:laxpsdslice}

We have seen in \Cref{sec:2limits} that $2$-terminal objects in the strict-slice of $2$-cones over a $2$-functor do not, in general, succeed in capturing both aspects of the universal properties of $2$-limits. In particular, the problem is that the strict-slice of $2$-cones does not see the modifications between two $2$-cones with same summit. In attempt to rectify this, we might consider richer slice $2$-categories containing more data in their morphisms: the \emph{lax-slice} and the \emph{pseudo-slice} of $2$-cones. However, $2$-terminal objects in these new slice $2$-categories seem to be unrelated to $2$-limits. As we present below, there are $2$-limits that are not $2$-terminal objects in the lax-slice (resp.~pseudo-slice), and conversely so.

We start by introducing lax- and pseudo-natural transformations between $2$-functors, and modifications between them.

\begin{definition} \label{def:nattransf}
Let $I$ and $\cA$ be $2$-categories, and let $F,G\colon I\to \cA$ be $2$-functors between them. A \textbf{lax-natural transformation} $\mu\colon F\Rightarrow G$ comprises the data of
\begin{rome}
\item a morphism $\ainl {Fi}{\mu_i}{Gi}$, for each $i\in I$,
\item a $2$-morphism $\ninl {(Gf)\mu_{i}}{\mu_f}{\mu_{j}(Ff)}$, for each morphism $f\colon i\to j$ in $I$,
\begin{diagram*}
\node(1)[]{$Fi$};
\node(2)[right of=1]{$Gi$};
\node(3)[below of=1]{$Fj$};
\node(4)[right of=3]{$Gj$};
\draw[a](1) to node[la]{$\mu_i$}(2);
\draw[a](1) to node[swap,la]{$Ff$}(3);
\draw[a](2) to node[la]{$Gf$}(4);
\draw[a](3) to node[swap,la]{$\mu_j$}(4);
\cell[la]{2}{3}{$\mu_f$};
\end{diagram*}
\end{rome}
which satisfy the following conditions:
\begin{enumerate}
    \item for all $i\in I$, $\mu_{\id_i}=\id_{\mu_i}$,
    \item for all composable morphisms $f,g$ in $I$, $\mu_{gf}=(\mu_g*Ff)(Gg*\mu_f)$,
    \item for all $2$-morphisms $\alpha\colon f\Rightarrow g$ in $I$, we have that $\mu_g(G\alpha*\mu_i)=(\mu_j*F\alpha)\mu_f$.
\begin{diagram*}
  \node(1)[]{$Fi$};
  \node(2)[right of=1]{$Gi$};
  \node(3)[below of=1]{$Fj$};
  \node(4)[right of=3]{$Gj$};
  \draw[a](1) to node[la]{$\mu_i$}(2);
  \draw[a](1) to node[swap,la]{$Fg$}(3);
  \draw[a, bend left=60](2) to node(f)[la]{$Gf$}(4);
  \draw[a](2) to node(g)[swap,la,xshift=.1cm]{$Gg$}(4);
  \draw[a](3) to node[swap,la]{$\mu_j$}(4);
  \cell[swap,la]{2}{3}{$\mu_g$};
  \cell[la,swap][n][0.55][.2cm]{f}{g}{$G\alpha$};

  \node(1)[right of=2,xshift=2cm]{$Fi$};
  \node(2)[right of=1]{$Gi$};
  \node(3)[below of=1]{$Fj$};
  \node(4)[right of=3]{$Gj$};
  \draw[a](1) to node[la]{$\mu_i$}(2);
  \draw[a](1) to node(f')[la,xshift=-.1cm]{$Ff$}(3);
  \draw[a, bend right=60](1) to node(g')[swap,la]{$Fg$}(3);
  \draw[a](2) to node[la]{$Gf$}(4);
  \draw[a](3) to node[swap,la]{$\mu_j$}(4);
  \cell[la]{2}{3}{$\mu_f$};
  \cell[la,swap][n][0.5][.2cm]{f'}{g'}{$F\alpha$};

  \node at ($(f)!0.5!(g')$) {$=$};
\end{diagram*}
\end{enumerate}

A \textbf{pseudo-natural transformation} is a lax-natural transformation $\ninl F\mu G$ whose every $2$-morphism component $\mu_f$ is invertible.
\end{definition}

\begin{definition} \label{def:modif}
Let $\ainl I{F,G}\cA$ be $2$-functors, and let $\ninl F{\mu,\nu}G$ be lax-natural transformations between them. A \textbf{modification} $\minl{\mu}{\varphi}{\nu}$ comprises the data of a $2$-morphism $\ninl{\mu_i}{\varphi_i}{\nu_{i}}$ for each $i\in I$, which satisfy $\nu_f(Gf*\varphi_i)=(\varphi_j*Ff)\mu_f$, for all morphisms $\ainl ifj$ in $I$.
\begin{diagram*}
  \node(1)[]{$Fi$};
  \node(2)[right of=1]{$Gi$};
  \node(3)[below of=1]{$Fj$};
  \node(4)[right of=3]{$Gj$};
  \draw[a](1) to node(g)[la,swap]{$\nu_i$}(2);
  \draw[a](1) to node(a)[la,swap]{$Ff$}(3);
  \draw[a, bend left=60,looseness=1.5](1) to node(f)[la]{$\mu_i$}(2);
  \draw[a](2) to node[la]{$Gf$}(4);
  \draw[a](3) to node[swap,la]{$\nu_j$}(4);
  \cell[la]{2}{3}{$\nu_f$};

  \coordinate (n) at ($(f)-(0.1cm,0)$);
  \coordinate (m) at ($(g)-(0.1cm,0)$);
  \cell[la,xshift=1pt][n]{n}{m}{$\varphi_{i}$};

  \node(1)[right of=2, xshift=1cm]{$Fi$};
  \node at ($(1)!0.5!(4)$) {$=$};
  \node(2)[right of=1]{$Gi$};
  \node(3)[below of=1]{$Fj$};
  \node(4)[right of=3]{$Gj$};
  \draw[a](1) to node[la]{$\mu_i$}(2);
  \draw[a](1) to node[swap,la]{$Ff$}(3);
  \draw[a, bend right=60,looseness=1.5](3) to node(g)[la,swap]{$\nu_j$}(4);
  \draw[a](2) to node[la]{$Gf$}(4);
  \draw[a](3) to node(f)[la]{$\mu_j$}(4);
  \cell[swap,la]{2}{3}{$\mu_f$};
  \coordinate (n) at ($(f)-(0.1cm,0)$);
  \coordinate (m) at ($(g)-(0.1cm,0)$);
  \cell[la,xshift=1pt][n]{n}{m}{$\varphi_{j}$};
\end{diagram*}

Similarly, we have a notion of modification between pseudo-natural transformations.
\end{definition}

\begin{remark}
  Note that a $2$-natural transformation $\ninl F\mu G$ as defined in \Cref{def:2nat} is precisely a lax-natural transformation whose every $2$-morphism component $\mu_f$ is an identity. Moreover, modifications in the sense just defined between two lax-natural transformations which happen to be $2$-natural coincide with the modifications of \Cref{def:strictmod}.
\end{remark}

As in the case of $2$-natural transformations, lax- and pseudo-natural transformations and modifications assemble into $2$-categories whose objects are $2$-functors.

\begin{notation}
  Let $I$ and $\cA$ be $2$-categories. We can define two $2$-categories whose objects are the $2$-functors $I\to \cA$:
\begin{rome}
\item the $2$-category $\Lax[I,\cA]$, whose $1$- and $2$-morphisms are
  lax-natural transformations and modifications,
\item the $2$-category $\Psd[I,\cA]$, whose $1$- and $2$-morphisms are
  pseudo-natural transformations and modifications.
\end{rome}
\end{notation}

The lax-slice and pseudo-slice of $2$-cones over a $2$-functor $\ainl IF\cA$ can be defined as pullbacks, as in \Cref{def:strictslice}, where we replace the upper-left corner with the $2$-categories $\Lax[\mathbbm{2},[I,\cA]]$ and $\Psd[\mathbbm{2},[I,\cA]]$, respectively. These constructions do not change the objects of the slice, but add more morphisms between them.

\begin{definition} \label{def:slice} Let $F\colon I\to \cA$ be a
  $2$-functor. The \textbf{lax-slice} $\lslice \Delta F$ of $2$-cones over
  $F$ is defined to be the following pullback in the ($1$-)category of
  $2$-categories and $2$-functors.
  \begin{diagram*}
    \node(1)[]{$\lslice \Delta F$};
    \node(2)[below of= 1]{$\cA\times \mathbbm{1}$};
    \node(3)[right of= 2, xshift=1.5cm]{$[I,\cA]\times [I,\cA]$};
    \node(4)[above of= 3]{$\Lax[\mathbbm{2},[I,\cA]]$};
    \draw[a](1)to (2);
    \draw[a](2)to node[swap,la]{$(\Delta,F)$}(3);
    \draw[a](1)to (4);
    \draw[a](4) to node[la]{$(s,t)$}(3);
    \pullback{1};
  \end{diagram*}
  This $2$-category $\lslice \Delta F$ is given by the following data:
  \begin{rome}
  \item an object in $\lslice \Delta F$ is a pair $(X,\mu)$ of an
    object $X\in \cA$ together with a $2$-natural transformation
    $\mu\colon \Delta X\Rightarrow F$,
  \item a morphism $(f,\varphi)\colon (X,\mu)\to (Y,\nu)$ consists of
    a morphism $f\colon X\to Y$ in $\cA$ together with a modification $\minl{\nu\Delta f}{\varphi}{\mu}$,
    \begin{diagram*}[node distance=1.7cm]
      \node(1)[]{$\Delta X$};
      \node(2)[right of= 1]{$\Delta Y$};
      \node(3)[below of= 2]{$F$};
      \draw[n](1)to node[la]{$\Delta f$}(2);
      \draw[n](2)to node[la]{$\nu$}(3);
      \draw[n](1)to node(a)[swap,la]{$\mu$}(3);
      \cell[la][t]{2}{a}{$\varphi$};
    \end{diagram*}
  \item a $2$-morphism $\alpha\colon (f,\varphi)\Rightarrow (g,\psi)$ between morphisms $(f,\varphi),(g,\psi)\colon (X,\mu)\to (Y,\nu)$ is a $2$-morphism $\alpha\colon f\Rightarrow g$ in $\cA$ such that $\psi (\nu*\Delta \alpha)=\varphi$.
    \begin{diagram*}[node distance=1.7cm]
      \node(4)[]{$\Delta X$};
      \node(5)[right of= 4]{$\Delta Y$};
      \node(6)[below of= 5]{$F$};
      \node(1)[right of=5]{$\Delta X$};
      \node(2)[right of= 1]{$\Delta Y$};
      \node(3)[below of= 2]{$F$};
      \draw[n](1)to node[la]{$\Delta f$}(2);
      \draw[n](2)to node[la]{$\nu$}(3);
      \draw[n](1)to node(a)[swap,la]{$\mu$}(3);
      \cell[la][t]{2}{a}{$\varphi$};
      \draw[n,bend left=65,looseness=1.5](4)to node(t)[la]{$\Delta f$}(5);
      \draw[n,bend right=0](4)to node(b)[swap,la]{$\Delta g$}(5);
      \draw[n](5)to node[la]{$\nu$}(6);
      \draw[n](4)to node(c)[swap,la]{$\mu$}(6);

      \coordinate (n) at ($(t)-(0.2cm,0)$);
      \coordinate (m) at ($(b)-(0.2cm,0)$);
      \cell[la,xshift=1pt][t]{n}{m}{$\Delta\alpha$};

      \coordinate (p) at ($(5)-(0,0.1cm)$);
      \coordinate (q) at ($(c)-(0,0.1cm)$);
      \cell[la,xshift=-1pt][t][0.45]{p}{q}{$\psi$};

      \node at ($(6)!0.5!(1)$) {$=$};
    \end{diagram*}
  \end{rome}

  Similarly, we can define the \textbf{pseudo-slice} $\pslice \Delta F$ of $2$-cones over~$F$ by replacing the upper-left corner $\Lax[\mathbbm{2},[I,\cA]]$ in the
  pullback above with
  $\Psd[\mathbbm{2},[I,\cA]]$. The pseudo-slice corresponds to
  the sub-$2$-category of the lax-slice $\lslice \Delta F$ containing all objects and only the morphisms $(f,\varphi)$ for which the
  modification $\varphi$ is invertible, and which is locally-full on
  $2$-morphisms.
\end{definition}

\begin{remark}
  Note that the strict-slice $\sslice \Delta F$ as defined in \Cref{def:strictslice} corresponds to the locally-full sub-$2$-category of the lax- or pseudo-slice containing all objects and only the morphisms $(f,\varphi)$ for which the
  modification $\varphi$ is an identity.
\end{remark}

We now give two counter-examples which show that
\begin{itemize}
    \item not every $2$-limit is $2$-terminal in the lax-slice of $2$-cones (\cref{ce:2lnot2tl}),
    \item not every $2$-terminal object in the lax-slice of $2$-cones is a $2$-limit (\cref{ce:tlsnot2l}).
\end{itemize}
These statements imply that, unlike in the case of strict-slices, $2$-terminal objects in the lax-slice are not at all related to $2$-limits. We derive counter-examples to show that the same is true for pseudo-slices, namely that $2$-terminal objects in the pseudo-slice of $2$-cones are not related to $2$-limits.

We first give an example of a $2$-limit that is not $2$-terminal in the lax-slice of $2$-cones. To illustrate this failure we seek a case where the lax-slice sees too many morphisms between the $2$-cones. In the counter-example below, we show that a $2$-morphism that is part of the $2$-dimensional aspect of the universal property of a $2$-limit might create undesirable morphisms in the lax-slice of $2$-cones.

\begin{ce}\label{ce:2lnot2tl}
    Let $I$ be the pullback shape diagram $\{ \bullet \longrightarrow\bullet \longleftarrow \bullet\}$. Let $\cA$ be the $2$-category generated by the data
  \begin{diagram*}
    \node(1)[]{$X$};
    \node(2)[below right of= 1]{$L$};
    \node(3)[right of= 2]{$B$};
    \node(4)[below of= 2]{$C$};
    \node(5)[below of= 3]{$A$};
    \draw[a,bend left=40](1)to node(f)[la,shrink]{$f$}(2);
    \draw[a,bend right=40](1)to node(g)[la,swap]{$g$}(2);
    \draw[a](2)to node[la]{$\lambda_{0}$}(3);
    \draw[a](2)to node[swap,la]{$\lambda_{1}$}(4);
    \draw[a](4)to node[swap,la]{$c$}(5);
    \draw[a](3)to node[la]{$b$}(5);
    \cell[swap,la]{f}{g}{$\alpha$};
  \end{diagram*}
  subject to the relation
  $b\lambda_0=c\lambda_1$. Take $F\colon I\to \cA$ to be the diagram
  \begin{center}
  \begin{tikzpicture}[node distance=1.3cm, baseline=(0.base)]
      \node(1)[]{$B$};
      \node(0)[below of= 1]{$A$};
      \node(2)[left of= 0]{$C$};
      \draw[a] (2) to node[swap,la]{$c$} (0);
      \draw[a] (1) to node[la]{$b$} (0);
    \end{tikzpicture} .
    \end{center}
\end{ce}

\begin{claim}
  The object $(L,\ninl{\Delta L}{\lambda} F)$ is the
  $2$-limit of $F$, but it is not $2$-terminal in the lax-slice
  $\lslice \Delta F$ of $2$-cones over $F$.
\end{claim}

\begin{proof}
  Let us begin by enumerating all the $2$-cones over $F$:
  \[ (L,\lambda),\quad (X,\lambda* f), \ \  \text{and} \ \ (X,\lambda* g).
\]

  We can see that $(L,\lambda)$ is a $2$-limit of $F$, since we have
  \[ \cA(X,L)=\{f\overset{\alpha}{\Rightarrow}g\} \ \ \text{and} \ \
    [I,\cA](\Delta X,F)=\{{\lambda*f}
    \overset{\lambda*\alpha}{\threecellarr} {\lambda*g}\}\] and  $\cA(L,L)=\{\id_L\}$ and
  $[I,\cA](\Delta L,F)=\{\lambda\}$.

  However, there are two distinct morphisms from $(X,\lambda *g)$ to $(L,\lambda)$ in the lax-slice $\lslice \Delta F$, which are given by
  \begin{align*}
    (g,(\id_{\lambda_0g},\id_{\lambda_1g}))\colon &(X,\lambda*g)\to (L,\lambda) \\
    (f,(\lambda_0*\alpha,\lambda_1*\alpha)) \colon &(X,\lambda *g)\to (L,\lambda).
  \end{align*}
  Therefore $(L,\lambda)$ is not $2$-terminal in $\lslice \Delta F$.
\end{proof}

\begin{reduction}\label{rem:2lnottp}
  By requiring $\alpha$ to be invertible in \Cref{ce:2lnot2tl}, we can
  similarly show that $(L,\lambda)$ is the $2$-limit of $F$, but is
  not $2$-terminal in the pseudo-slice $\pslice \Delta F$ of $2$-cones over $F$.
\end{reduction}

Next we give an example of a $2$-terminal object in the lax-slice of
$2$-cones that is not a $2$-limit. This counter-example is designed to capture a particular arrangement of two $2$-cones over a $2$-functor together with a single non-trivial modification between them. This modification gives rise to a morphism in the lax-slice between these two $2$-cones, exhibiting the target $2$-cone as $2$-terminal in the lax-slice. However, the source $2$-cone is not in the image of the post-composition functor by the target $2$-cone, which shows that the latter is not a $2$-limit.

\begin{ce}\label{ce:tlsnot2l}
  Let $I$ be the pullback shape diagram $\{ \bullet \longrightarrow\bullet \longleftarrow \bullet\}$. Let $\cA$ be the $2$-category generated by the data
   \begin{diagram*}
     \node(1)[]{$L$};
     \node(2)[right of=1]{$B$};
     \node(3)[below of=1]{$C$};
     \node(4)[below of=2]{$A$};
     \node(5)[above left of= 1]{$X$};

     \draw[a](3)to node[swap,la]{$c$}(4);
     \draw[a](2)to node[]{$b$}(4);
     \draw[a](1)to node[la,over]{$\lambda_{0}$}(2);
     \draw[a](1)to node[la,over]{$\lambda_{1}$}(3);
     \draw[a](5)to node[la,over,shrink]{$f$}(1);
     \draw[a,bend right=30](5)to node(R)[la,swap]{$\alpha_{1}$}(3);
     \draw[a,bend left=30](5)to node(S)[la,shrink]{$\alpha_{0}$}(2);

     \cell[la,near start,swap][n][0.45]{1}{R}{$\gamma_{1}$};
     \cell[la,near start]{1}{S}{$\gamma_{0}$};
  \end{diagram*}
  subject to the relations $b\lambda_0=c\lambda_1$ and
  $b\alpha_0=c\alpha_1$. Take $F\colon I\to \cA$ to be the diagram
  \begin{center}
  \begin{tikzpicture}[node distance=1.3cm,baseline=(0.base)]
      \node(1)[]{$B$};
      \node(0)[below of= 1]{$A$};
      \node(2)[left of= 0]{$C$};
      \draw[a] (2) to node[swap,la]{$c$} (0);
      \draw[a] (1) to node[la]{$b$} (0);
    \end{tikzpicture} .
    \end{center}
\end{ce}

\begin{claim}
  The object $(L,\ninl{\Delta L}{\lambda} F)$ is $2$-terminal
  in the lax-slice $\lslice \Delta F$ of $2$-cones over~$F$, but the functor
  \[ \ainl{\cA(X,L)}{\lambda_*\circ\Delta}{[I,\cA](\Delta X,F)} \]
  given by post-composition with $\lambda$ is not surjective on
  objects, i.e.~$(L,\lambda)$ is not a $2$-limit of~$F$.
\end{claim}

\begin{proof}
  The objects of the lax-slice $\lslice \Delta F$ are given by the $2$-cones over $F$:
\[ (L,\lambda),\quad (X,\alpha), \ \  \text{and} \ \ (X,\lambda* f).
\]
Each of these objects admits precisely one morphism to $(L,\lambda)$
in $\lslice \Delta F$ given by
\begin{align*}
    (\id_L,\id_{\lambda_0},\id_{\lambda_1})\colon &(L,\lambda)\to (L,\lambda) \\
    (f,\gamma_0,\gamma_1)\colon &(X,\alpha)\to (L,\lambda) \\
    (f,\id_{\lambda_0 f},\id_{\lambda_1 f})\colon &(X,\lambda* f)\to (L,\lambda).
\end{align*}
There are no non-trivial $2$-morphisms to $(L,\lambda)$ in
$\sslice \Delta F$, since there are no non-trivial $2$-morphisms between
$X$ and $L$ in $\cA$. This proves that $(L,\lambda)$ is $2$-terminal
in $\lslice \Delta F$.

However, the $2$-cone $\ninl{\Delta X}{\alpha}F$ is an object of
$[I,\cA](\Delta X,F)$, but it is not in the image of
$\lambda_*\circ\Delta$.
\end{proof}

\begin{reduction}\label{rem:2tpnot2l}
  By requiring $\gamma_0$ and $\gamma_1$ to be invertible in \Cref{ce:tlsnot2l}, we can
  similarly show that $(L,\lambda)$ is $2$-terminal in the pseudo-slice $\pslice \Delta F$ of $2$-cones
  over $F$, but it is not a $2$-limit of $F$.
\end{reduction}

\section{The cases of pseudo- and lax-limits} \label{sec:psdlaxlimit}

Recall that part of the definition of a $2$-limit involved $2$-cones which were $2$-natural transformations. In \Cref{sec:laxpsdslice}, we presented weaker notions of $2$-dimensional natural transformations, namely pseudo- and lax-natural transformations. These give other ways of taking a $2$-dimensional limit of a $2$-functor, by changing the shape of the $2$-dimensional cones. The corresponding notions are called \emph{pseudo-limits} and \emph{lax-limits}.

In this section, we show that results similar to those we have seen in \Cref{sec:2limits,sec:laxpsdslice} hold for pseudo- and lax-limits. In the lax-limit case, we show that:
\begin{itemize}
    \item every lax-limit is $2$-terminal in the strict-slice of lax-cones (\Cref{rem:pllltss}),
    \item not every $2$-terminal object in the strict-slice of lax-cones is a lax-limit (\cref{ce:2tssnotll}),
    \item not every lax-limit is $2$-terminal in the lax-slice of lax-cones (\cref{ce:llnot2tls}),
    \item not every $2$-terminal object in the lax-slice of lax-cones is a lax-limit (\cref{ce:2tlnotll}).
\end{itemize}
From the last two, we also derive the result that lax-limits are not related to $2$-terminal objects in the pseudo-slice of lax-cones. With all of the results and counter-examples we have established thus far, we are able to derive proofs and counter-examples covering the conjectures related to pseudo-limits.

We first introduce the notions of pseudo- and lax-limits.

\pagebreak

\begin{definition} \label{def:laxlim}
Let $I$ and $\cA$ be
  $2$-categories, and let $F\colon I\to \cA$ be a $2$-functor.
  \begin{rome}
  \item A \textbf{pseudo-limit} of $F$ comprises the data of an object $L\in \cA$
  together with a pseudo-natural transformation
  $\lambda\colon \Delta L\Rightarrow F$ such that, for each
  object $X\in\cA$, the functor
  \[ \ainl{\cA(X,L)}{\lambda_*\circ\Delta}{\Psd[I,\cA](\Delta X, F)} \]
  given by post-composition with $\lambda$ is an isomorphism of categories.
  \item A \textbf{lax-limit} of $F$ comprises the data of an object $L\in \cA$
  together with a lax-natural transformation
  $\lambda\colon \Delta L\Rightarrow F$ such that, for each
  object $X\in\cA$, the functor
  \[ \ainl{\cA(X,L)}{\lambda_*\circ\Delta}{\Lax[I,\cA](\Delta X, F)} \]
  given by post-composition with $\lambda$ is an isomorphism of categories.
  \end{rome}
\end{definition}

In order to consider slices in which these pseudo- and lax-limit cones live, we need to change the shape of the cone objects of the slices considered in \Cref{def:strictslice,def:slice}.

\begin{remark}
  We can also define the strict-, pseudo-, and lax-slices of pseudo-cones
  (resp. lax-cones) over $F$, by considering objects of the form $(X,\mu)$ where $\ninl {\Delta X}\mu F$ is a pseudo-natural (resp.~lax-natural) transformation. These constructions can be achieved by replacing $[I,\cA]$ in the pullbacks of
  \Cref{def:strictslice,def:slice} with $\Psd[I,\cA]$ (resp.~$\Lax[I,\cA]$). For example, the pseudo-slice of lax-cones is the following pullback.
  \begin{diagram*}
    \node(1)[]{$\pslice[lx]{\Delta}{F}$};
    \node(2)[below of= 1]{$\cA\times \mathbbm{1}$};
    \node(3)[right of= 2, xshift=2cm]{$\Lax[I,\cA]\times \Lax[I,\cA]$};
    \node(4)[above of= 3]{$\Psd[\mathbbm{2},\Lax[I,\cA]]$};
    \draw[a](1)to (2);
    \draw[a](2)to node[swap,la]{$(\Delta,F)$}(3);
    \draw[a](1)to (4);
    \draw[a](4) to node[la]{$(s,t)$}(3);
    \pullback{1};
  \end{diagram*}
\end{remark}

There is an analogue of \Cref{pos:2l2ts} in the case of pseudo-limits (resp.~lax-limits), whose proof may be derived by replacing $2$-natural transformations with pseudo-natural ones (resp.~lax-natural ones).

\begin{proposition}\label{rem:pllltss}
A pseudo-limit (resp.~lax-limit) of a $2$-functor is $2$-terminal in the strict-slice of pseudo-cones (resp.~lax-cones).
\end{proposition}

However, not every $2$-terminal object in the strict-slice of pseudo- or lax-cones is a pseudo- or lax-limit. In particular, \Cref{ce:2tsnot2l} exhibits such an object in the pseudo-limit case:

\begin{reduction}\label{rem:tsnotllpl}
    Since there are no invertible $2$-morphisms in \Cref{ce:2tsnot2l}, this is also an example of a $2$-terminal object in the strict-slice of pseudo-cones that is not a pseudo-limit.
\end{reduction}

Let us recall that \Cref{ce:2tsnot2l} has $2$-morphisms $\gamma_0$ and $\gamma_1$ that introduce two additional lax-cones with summit $X$ over $F$. These new lax-cones do not admit a morphism to $(L,\lambda)$ in the strict-slice of lax-cones. This arrangement demonstrates that $(L,\lambda)$ is \emph{not} $2$-terminal in the strict-slice of lax-cones. Therefore, we can not use this counter-example for the lax-limit case.

The issue at heart here is that modifications between lax-cones over a $2$-functor may be turned into lax-cones over this same $2$-functor. Thus, to find an example of a $2$-terminal object in the strict-slice of lax-cones that is not a lax-limit, we must find a case where such a transformation is not possible. In order to create such an example, we need the diagram shape to have objects that are both the source and the target of a non-trivial morphism.

\begin{ce} \label{ce:2tssnotll}
  Let $I$ be the $2$-category freely generated by the data
  \begin{center}
  \begin{tikzpicture}[node distance=1.3cm,baseline=(0.base)]
      \node(0)[]{$0$};
      \node(1)[right of= 0]{$1$};
      \draw[a] ($(0.east)+(0,3pt)$) to node[la]{$x$} ($(1.west)+(0,3pt)$);
      \draw[a] ($(1.west)-(0,3pt)$) to node[la]{$y$} ($(0.east)-(0,3pt)$);
    \end{tikzpicture} ,
    \end{center}
    i.e.~the non-trivial morphisms in $I$ are given by all possible composites of $x$ and $y$, e.g.~$xyxy$.
    Let $\cA$ be the $2$-category freely generated by the data
    \begin{diagram*}[node distance=1.2cm]
    \node(1)[]{$X$};
    \node(2)[above right of= 1,xshift=.3cm]{$L$};
    \node(3)[below right of= 1,xshift=.3cm]{$L$};
    \node(4)[below right of= 2,xshift=.3cm]{$A$};
    \node(6)[below right of= 3,xshift=.3cm]{$B$};
    \draw[a](1)to node[la]{$f$}(2);
    \draw[a](1)to node[swap,la]{$g$}(3);
    \draw[a](2)to node[la]{$\lambda_{0}$}(4);
    \draw[a](3)to node[la,over,shrink]{$\lambda_{0}$}(4);
    \draw[a](3)to node(b)[swap,la]{$\lambda_{1}$}(6);
    \draw[a](4)to node[la]{$a$}(6);
    \cell[la]{2}{3}{$\gamma_{0}$};
    \cell[la, near end]{4}{b}{$\lambda_x$};
    \node(7)[right of= 3,xshift=1cm]{$X$};
    \node(8)[above right of= 7,xshift=.3cm]{$L$};
    \node(9)[below right of= 7,xshift=.3cm]{$L$};
    \node(10)[below right of= 8,xshift=.3cm]{$B$};
    \node(12)[above right of= 8,xshift=.3cm]{$A$};
    \draw[a](7)to node[la]{$f$}(8);
    \draw[a](7)to node[swap,la]{$g$}(9);
    \draw[a](8)to node[over,la,shrink]{$\lambda_{1}$} node(b)[swap]{} (10);
    \draw[a](9)to node[la,swap]{$\lambda_{1}$}(10);
    \draw[a](8)to node[la]{$\lambda_{0}$}(12);
    \draw[a](12)to node(a)[la]{$a$}(10);
    \cell[la]{8}{9}{$\gamma_{1}$};
    \cell[la,near end,xshift=-2pt]{12}{b}{$\lambda_x$};
    \node at ($(4)!0.5!(7)$) {$=$};

    \node(1)[right of =a]{$X$};
    \node(2)[above right of= 1,xshift=.3cm]{$L$};
    \node(3)[below right of= 1,xshift=.3cm]{$L$};
    \node(4)[below right of= 2,xshift=.3cm]{$B$};
    \node(6)[below right of= 3,xshift=.3cm]{$A$};
    \draw[a](1)to node[la]{$f$}(2);
    \draw[a](1)to node[swap,la]{$g$}(3);
    \draw[a](2)to node[la]{$\lambda_{1}$}(4);
    \draw[a](3)to node[la,over,shrink]{$\lambda_{1}$}(4);
    \draw[a](3)to node(b)[swap,la]{$\lambda_{0}$}(6);
    \draw[a](4)to node[la]{$b$}(6);
    \cell[la]{2}{3}{$\gamma_{1}$};
    \cell[la, near end]{4}{b}{$\lambda_y$};
    \node(7)[right of= 3,xshift=1cm]{$X$};
    \node(8)[above right of= 7,xshift=.3cm]{$L$};
    \node(9)[below right of= 7,xshift=.3cm]{$L$};
    \node(10)[below right of= 8,xshift=.3cm]{$A$};
    \node(12)[above right of= 8,xshift=.3cm]{$B$};
    \draw[a](7)to node[la]{$f$}(8);
    \draw[a](7)to node[swap,la]{$g$}(9);
    \draw[a](8)to node[over,la,shrink]{$\lambda_{0}$} node(b)[swap]{} (10);
    \draw[a](9)to node[la,swap]{$\lambda_{0}$}(10);
    \draw[a](8)to node[la]{$\lambda_{1}$}(12);
    \draw[a](12)to node[la]{$b$}(10);
    \cell[la]{8}{9}{$\gamma_{0}$};
    \cell[la,near end,xshift=-2pt]{12}{b}{$\lambda_y$};
    \node at ($(4)!0.5!(7)$) {$=$};
    \end{diagram*}
    subject to the relations $(\lambda_x*g)(a*\gamma_0)=\gamma_1(\lambda_x*f)$, and $(\lambda_y*g)(b*\gamma_1)=\gamma_0(\lambda_y*f)$. Again, we have all possible composites of the morphisms $a$ and $b$ in $\cA$, and all possible pastings of the $2$-morphisms $\lambda_x$ and $\lambda_y$. Take $F\colon I\to \cA$ to be the diagram defined on the generators $x$ and $y$ of $I$ by
  \begin{center}
  \begin{tikzpicture}[node distance=1.3cm,baseline=(0.base)]
      \node(0)[]{$A$};
      \node(1)[right of= 0]{$B$};
      \draw[a] ($(0.east)+(0,3pt)$) to node[la]{$a$} ($(1.west)+(0,3pt)$);
      \draw[a] ($(1.west)-(0,3pt)$) to node[la]{$b$} ($(0.east)-(0,3pt)$);
    \end{tikzpicture} .
    \end{center}
\end{ce}

\begin{remark}
  Note that the morphisms $\lambda_0\colon L\to A$ and $\lambda_1\colon L\to B$ together with the $2$-morphisms $\lambda_x\colon a\lambda_0\Rightarrow \lambda_1$ and $\lambda_y\colon b\lambda_1\Rightarrow \lambda_0$ suffice to give the data of a lax-natural transformation $\lambda\colon \Delta L\Rightarrow F$. Indeed, by \Cref{def:nattransf} (2), the $2$-morphism component of $\lambda$ at some composite of $x$ and $y$ is determined by the corresponding pasting of the $2$-morphism components $\lambda_x$ and $\lambda_y$.
\end{remark}

\begin{claim}
  The object $(L,\lambda\colon \Delta L\Rightarrow F)$ is $2$-terminal in the strict-slice $\sslice[lx] \Delta F$ of lax-cones over~$F$, but the functor
\[ \ainl{\cA(X,L)}{\lambda_*\circ\Delta}{\Lax[I,\cA](\Delta X,F)} \] given
by post-composition with $\lambda$ is not surjective on morphisms,
i.e.~$(L,\lambda)$ is not a lax-limit of $F$.
\end{claim}

\begin{proof}
  The objects of the strict-slice $\sslice[lx] \Delta F$ are given by the
  lax-cones over $F$:
\[ (L,\lambda),\quad (X,\lambda*f), \ \  \text{and} \ \ (X,\lambda*g).
\]
Note that the $2$-morphisms $\gamma_0$ and $\gamma_1$ do not induce lax-cones over $F$ with summit $X$, since there are no $2$-morphisms from $\lambda_1 g$ to $\lambda_0 f$, and from $\lambda_0 g$ to $\lambda_1 f$ in $\cA$, respectively. There are also no lax-cones over $F$ with summit $A$ or $B$ since there are no non-trivial $2$-morphisms between any two composites of $a$ and $b$ in $\cA$.
Each of the objects above admits precisely one morphism to $(L,\lambda)$
in $\sslice[lx] \Delta F$ given by
\begin{align*}
    \id_L\colon &(L,\lambda) \to (L,\lambda) \\
    f\colon &(X,\lambda*f)\to (L,\lambda) \\
    g\colon &(X,\lambda*g)\to (L,\lambda).
\end{align*}
There are no non-trivial $2$-morphisms to $(L,\lambda)$ in
$\sslice[lx] \Delta F$, since there are no non-trivial $2$-morphisms between $X$ and $L$ in $\cA$. This proves that $(L,\lambda)$ is
$2$-terminal in $\sslice[lx] \Delta F$.

However, the $2$-morphisms $\gamma_0$ and $\gamma_1$ give the data of
a modification $\minl{\lambda*f}{\gamma}{\lambda*g}$, i.e.~a morphism in
$\Lax[I,\cA](\Delta X, F)$. But there is no $2$-morphism between $f$ and $g$
in $\cA$ that maps to $\gamma$ via
$\ainl{\cA(X,L)}{\lambda_*\circ\Delta}{\Lax[I,\cA](\Delta X, F)}$. Hence
$(L,\lambda)$ is not the lax-limit of $F$.
\end{proof}

\begin{remark}
  Note that, in the above counter-example, it is essential to have all free composites of $x$ and $y$ in $I$, and also of $a$ and $b$ in \cA. Indeed, if we impose any conditions on the composites of $x$ and $y$, e.g.~$x$ and $y$ are mutual inverses, then the $2$-functor $F$ must preserve these, and these new conditions on $a$ and $b$ in \cA add undesirable lax-cones with summits $A$ and~$B$.
\end{remark}

However, when the $2$-category $\cA$ admits tensors by $\mathbbm{2}$, there is an analogue of \Cref{prop:tensor} in the pseudo-limit (resp.~lax-limit) case, whose proof may be derived by replacing $2$-natural transformations by pseudo-natural ones (resp.~lax-natural ones).

\begin{proposition}
 Suppose $\cA$ is a $2$-category that admits tensors by $\mathbbm{2}$, and let $F\colon I\to \cA$ be a $2$-functor. Then an object is $2$-terminal in the strict-slice of pseudo-cones (resp.~lax-cones) over $F$ if and only if it is a pseudo-limit (resp.~lax-limit) of~$F$.
\end{proposition}

We dedicate the rest of the section to exploring counter-examples which together refute all the remaining conjectures relating pseudo- and lax-limits to the lax- and pseudo-slices of appropriate cones. We begin by recalling \Cref{ce:2lnot2tl}, which also shows that not every pseudo-limit is a $2$-terminal object in the lax-slice of pseudo-cones:

\begin{reduction}\label{red:plnot2tl}
    Since there are no invertible $2$-morphisms in \Cref{ce:2lnot2tl}, this is also an example of a pseudo-limit that is not $2$-terminal in the lax-slice of pseudo-cones.
\end{reduction}

Let us recall that \Cref{ce:2lnot2tl} has a $2$-morphism $\alpha$. This $2$-morphism was an obstruction to $(L,\lambda)$ being $2$-terminal in the lax-slice, but not to being a $2$-limit. In the move to lax-cones, however, this $2$-morphism introduces an additional lax-cone over $F$ with summit $X$ that is not in the image of $\lambda_*\circ \Delta$. Therefore, $(L,\lambda)$ cannot be a lax-limit of~$F$, and we need a new counter-example for the lax-limit case.

Our new counter-example should have a non-trivial $2$-morphism to serve as an obstruction to $2$-terminality in the lax-slice. But, to ensure that this $2$-morphism is in the image of post-composition by the to-be lax-limit cone, we must also introduce new relations.

\pagebreak

\begin{ce}\label{ce:llnot2tls}
  Let $I=\mathbbm 2$, and let $\cA$ be the
  $2$-category generated by the data
  \begin{diagram*}
    \node(1)[]{$X$};
    \node(2)[right of= 1]{$A$};
    \node(3)[right of= 2]{$B$};
    \draw[a,bend left=40](1)to node(t)[la]{$\alpha_{0}$}(2);
    \draw[a,bend right=40](1)to node(b)[la,swap]{$\alpha_{1}$}(2);
    \draw[a,la](2)to node[la]{$f$}(3);
    \cell[la]{t}{b}{$\alpha$};
  \end{diagram*}
  subject to the relations $f\alpha_{0}=f\alpha_{1}$ and
  $f*\alpha=\id_{f\alpha_{0}}$. Consider the $2$-functor $\ainl{\mathbbm 2}{f}\cA$ given by the
  morphism $\ainl AfB$.
\end{ce}

\begin{claim}
  The object $(A,\ninl{\Delta A}{\id_f}{f})$ is the lax-limit of $f$
  in \cA, but it is not $2$-terminal in the lax-slice $\lslice[lx] \Delta f$ of
  lax-cones over $f$.
\end{claim}

\begin{proof}
  Let us begin by enumerating all the lax-cones over $f$:
  \[ (A,\id_{f}),\quad
    (X,\id_{f\alpha_{0}}),\ \ \text{and}\ \
    (X,\id_{f\alpha_{1}})\ .\]
  Note that the last two above-listed objects differ in their cone components to $A$: in the first case the leg is $\alpha_0$, and in the second case it is $\alpha_1$.

  We can see that $(A,\id_f)$ is a lax-limit of $f$, since we have
  \[ \cA(X,A)=\{\alpha_0\overset{\alpha}{\Rightarrow}\alpha_1\} \ \ \text{and} \ \
    \Lax[\mathbbm{2},\cA](\Delta X,f)=\{{\id_{f\alpha_0}}
    \overset{f*\alpha}{\threecellarr} {\id_{f\alpha_1}}\}\] and  $\cA(A,A)=\{\id_A\}$ and
  $[\mathbbm{2},\cA](\Delta A,f)=\{\id_f\}$.

  However, there are two distinct morphisms from $(X,\id_{f\alpha_{1}})$ to $(A,\id_{f})$ in the lax-slice $\lslice[lx] \Delta f$ of lax-cones, given by
  \begin{align*}
    (\alpha_{1},(\id_{\alpha_{1}},\id_{f\alpha_{1}}))\colon & (X,\id_{f\alpha_{1}}) \to  (A,\id_{f})\\
    (\alpha_{0},(\alpha,\id_{f\alpha_1})) \colon & (X,\id_{f\alpha_{1}}) \to  (A,\id_{f}).
  \end{align*}
  where we have used the fact $f*\alpha=\id_{f\alpha_{1}}$ in displaying the latter morphism. Therefore, $(A,\id_f)$ is not $2$-terminal in the lax-slice of lax-cones over $f$.
\end{proof}

\begin{remark}
  Note that the object $(A,\id_f)$ is also a pseudo-limit of $f$. This gives a second example of a pseudo-limit that is not $2$-terminal in the lax-slice of pseudo-cones.
\end{remark}

\begin{reduction} \label{red:llplnot2tp}
  By requiring $\alpha$ to be invertible in \Cref{ce:llnot2tls}, we can
  similarly show that $(A,\id_{f})$ is a lax-limit (resp.~pseudo-limit) of $f$, which is
  not $2$-terminal in the pseudo-slice of lax-cones (resp.~pseudo-cones).
\end{reduction}

Moreover, one can derive from \Cref{ce:tlsnot2l} that not every $2$-terminal object in the lax-slice of pseudo-cones is a pseudo-limit:

\begin{reduction}\label{rem:pseudobad}
  Since there are no invertible $2$-morphisms in \Cref{ce:tlsnot2l}, this is also an example of a $2$-terminal object in the lax-slice of
  pseudo-cones that is not a pseudo-limit.
\end{reduction}

\Cref{ce:tlsnot2l} in fact also applies in the lax-cone case. Although
the computations are more involved, one can check this is also an example of a $2$-terminal object in the lax-slice of lax-cones that it is not a lax-limit. However, there is a more striking example for this case. When
considering diagrams of shape $\mathbbm{2}$, it turns out that even a
single non-trivial lax-cone over a morphism exhibits a $2$-terminal
object in the lax-slice of lax-cones that is not a lax-limit of the morphism.

\begin{ce}\label{ce:2tlnotll}
  Let $I=\mathbbm{2}$, and let $\cA$ be the
  $2$-category generated by the data
  \begin{center}
  \begin{tikzpicture}[baseline=(3.base)]
    \node(1)[]{$X$};
    \node(2)[right of= 1]{$A$};
    \node(3)[below of= 2]{$B$};
    \draw[a](1)to node[la]{$\alpha_{0}$}(2);
    \draw[a](2)to node[la]{$f$}(3);
    \draw[a](1)to node(a)[swap,la]{$\alpha_{1}$}(3);
    \cell[la]{2}{a}{$\alpha$};
  \end{tikzpicture} .
  \end{center}
  Consider the $2$-functor $\ainl{\mathbbm 2}{f}\cA$ given by the morphism
  $\ainl{A}{f}{B}$.
\end{ce}

\begin{claim}
  The object $(A,\ninl{\Delta A}{\id_f}{f})$ is $2$-terminal
  in the lax-slice $\lslice[lx]\Delta f$ of lax-cones over~$f$, but the
  functor
  \[ \ainl{\cA(X,A)}{(\id_{f})_*\circ\Delta}{\Lax[\mathbbm{2},\cA](\Delta X,f)} \]
  is not surjective on objects, i.e.~$(A,\id_f)$ is not a lax-limit
  of $f$.
\end{claim}

\begin{proof}
  The objects of the lax-slice $\lslice[lx] \Delta f$ are the following
  lax-cones over~$f$:
  \[ (A,\id_f),\quad (X,\id_{f\alpha_0}), \ \ \text{and} \ \
    (X,\alpha).
  \]
  Each of these objects admits precisely one morphism to $(A,\id_f)$
  in $\lslice[lx] \Delta f$, given by
  \begin{align*}
    (\id_A,(\id_A,\id_f))\colon & (A,\id_f)\to (A,\id_f) \\
    (\alpha_0,(\id_{\alpha_0},\id_{f\alpha_0}))\colon &(X,\id_{f\alpha_0})\to (A,\id_f)\\
    (\alpha_0,(\id_{\alpha_0},\alpha))\colon &(X,\alpha)\to (A,\id_f).
  \end{align*}
  As there are no non-trivial $2$-morphisms between $X$ and $A$ in \cA,
  there are no non-trivial $2$-morphisms to $(A,\id_f)$ in
  $\lslice[lx] \Delta f$. This shows that $(A,\id_f)$ is $2$-terminal in
  the lax-slice $\lslice[lx] \Delta f$ of lax-cones.

  Next, observe that the lax-cone
  $\ninl{f\alpha_{0}}{\alpha}{\alpha_{1}}$ is an object of
  $\Lax[I,\cA](\Delta X,f)$, but it is not in the image of
  $(\id_f)_*\circ\Delta$. Hence $(A,\id_f)$ is not a lax-limit of
  $f$.
\end{proof}

\begin{reduction} \label{ce:2tpnotpl}
  By requiring $\alpha$ to be
  invertible in \Cref{ce:2tlnotll}, we can similarly show that
  $(A,\id_f)$ is $2$-terminal in the pseudo-slice of lax-cones
  (resp.~pseudo-cones) over~$f$, but that it is not a lax-limit
  (resp.~pseudo-limit) of $f$.
\end{reduction}

\section{Bi-type limits for the completionist} \label{sec:bilimits}

At this point we have seen that $2$-terminal objects in all slices and $2$-dimensional limits do not generally align. In this last section, which we present for completeness, we address one final weakening of the central definitions we have thus far considered. In defining $2$-dimensional limits with various strengths of cones, we have always asked for an \emph{isomorphism} of categories to govern the universal property. However, we might seek to relax this requirement by asking instead that the relevant functor induces only an \emph{equivalence} of categories. This leads to the following definitions.

\begin{definition} \label{def:bilimit} Let $I$ and $\cA$ be $2$-categories, and let $F\colon I\to \cA$ be a $2$-functor. A \textbf{bi-limit} of $F$ comprises the data of an object $L\in \cA$ together with a $2$-natural transformation $\lambda\colon \Delta L\Rightarrow F$ which are such that, for each object $X\in\cA$, the functor
  \[ \ainl{\cA(X,L)}{\lambda_*\circ\Delta}{[I,\cA](\Delta X, F)} \]
  given by post-composition with $\lambda$ is an equivalence of categories.

  Similarly, we can define \textbf{pseudo-bi-limit} (resp.~\textbf{lax-bi-limit}) by replacing $[I,\cA]$ in the above with $\Psd[I,\cA]$ (resp.~$\Lax[I,\cA])$.
\end{definition}

\begin{remark} \label{rem:univbilim}
    The two aspects of a universal property of a bi-limit may be reformulated more explicitly by expanding the content of the equivalence of categories above. For every $X\in \cA$,
  \begin{enumerate}
      \item  for every $2$-cone $\mu\colon \Delta X\Rightarrow F$, there is a morphism $f\colon X\to L$ in $\cA$ and an invertible modification $\minl{\lambda\Delta f}{\varphi}{\mu}$,
      \item for all morphisms $f,g\colon X\to L$, and for every modification $\minl{\lambda\Delta f}{\Theta}{\lambda\Delta g}$, there is a unique $2$-morphism $\alpha\colon f\Rightarrow g$ in $\cA$ such that $\lambda*\Delta\alpha=\Theta$.
  \end{enumerate}
\end{remark}

\begin{definition} \label{def:biterminal}
  Let $\cA$ be a $2$-category. An object $L\in \cA$ is \textbf{bi-terminal} if for all $X\in \cA$ there is an equivalence
  of categories $\cA(X,L)\eqv \mathbbm{1}$.
\end{definition}

In formulating analogous conjectures for the bi-limit and bi-terminal cases, we should pay careful attention to \Cref{rem:univbilim} (1). Observe that, given a $2$-cone $\ninl {\Delta X}\mu F$, the $1$-dimensional aspect of the universal property of a bi-limit $(L,\lambda)$ gives only a morphism $(X,\mu)\to(L,\lambda)$ of the pseudo-slice of $2$-cones $\pslice \Delta F$. It is thus inappropriate to look at bi-terminal objects in the strict-slice of $2$-cones when attempting to recover a general bi-limit.

Much as was the case for $2$-limits, bi-limits are in general bi-terminal objects in the pseudo-slice of $2$-cones. In fact, all of the positive results of \Cref{sec:2limits} follows in this context. We defer all proofs to the paper \cite{tslil} which deals with such relationships in greater generality. The first result can be deduced from \cite[Corollary 7.22]{tslil} and the second is \cite[Corollary 7.25]{tslil}.

\begin{proposition} \label{prop:bibi}
  Let $I$ and $\cA$ be $2$-categories, and let $\ainl IF\cA$ be a $2$-functor. If $(L,\lambda\colon \Delta L\Rightarrow F)$ is a bi-limit of $F$, then $(L,\lambda)$ is bi-terminal in the pseudo-slice  $\pslice {\Delta} F$ of $2$-cones over $F$.
\end{proposition}

\begin{proposition} \label{prop:bitensor}
  Suppose $\cA$ is a $2$-category that admits tensors by $\mathbbm{2}$, and let $F\colon I\to \cA$ be a $2$-functor. Then an object is bi-terminal in the pseudo-slice $\pslice \Delta F$ of $2$-cones over $F$ if and only if it is a bi-limit of $F$.
\end{proposition}

\begin{remark} \label{prop:bibilaxpsd}
    \Cref{prop:bibi,prop:bitensor} also hold true when the $2$-cones are replaced by pseudo- and lax-cones.
\end{remark}

With every sunrise there is a sunset, and just as the positive results extended themselves to this weaker context, so too do the negative. Since isomorphims of categories are, in particular, equivalences of categories, a $2$-type limit or a $2$-terminal object is, in particular, a bi-type limit or a bi-terminal object. Moreover, an examination of all of the counter-examples and reductions referenced in the tables below shows that there are no non-trivial invertible $2$-morphisms in the $2$-categories involved. This allows us to deduce the following facts. First, the notions of $2$-type limits (resp.~$2$-terminal objects) and bi-type limits (resp.~bi-terminal objects) in each of these counter-examples coincide. Second, for each of \Cref{ce:2tsnot2l,ce:2tssnotll}, the strict-slice and pseudo-slice coincide. Therefore, all counter-examples and reductions for $2$-type limits and $2$-terminality seen in previous sections are also counter-examples and reductions for bi-type limits and bi-terminality, as summarised in the following tables.

\begin{table}[!ht]
  \centering
  \caption{Bi-type limits which are not bi-terminal}
  \begin{tabular}{*{3}{c|}c}
    $\to$ implies bi-terminal in $\downarrow$ & Bi-limit & Pseudo-bi-limit & Lax-bi-limit\\\hline
    Pseudo-slice & \valid{} \cref{prop:bibi} & \valid{} \cref{prop:bibilaxpsd} & \valid{} \cref{prop:bibilaxpsd} \\
    Lax-slice & \cref{ce:2lnot2tl} & \Cref{red:plnot2tl} & \cref{ce:llnot2tls}
  \end{tabular}
\end{table}

\begin{table}[!ht]
  \centering
  \caption{Bi-terminal objects which are not bi-type limits}
  \begin{tabular}{*{3}{c|}c}
    Bi-terminal in $\downarrow$ implies $\to$ & Bi-limit & Pseudo-bi-limit & Lax-bi-limit \\\hline
    Pseudo-slice & \cref{ce:2tsnot2l} & \cref{rem:tsnotllpl} & \cref{ce:2tssnotll} \\
    Lax-slice & \cref{ce:tlsnot2l} & \cref{rem:pseudobad} & \cref{ce:2tlnotll}
  \end{tabular}
\end{table}

\bibliographystyle{plain}
\bibliography{references}

\end{document}